\documentclass[10pt]{amsart}
\usepackage{amssymb,amsmath,amsthm}
\usepackage{mathrsfs,a4wide}
\usepackage{calligra}
\usepackage{epsfig}
\usepackage{appendix}
\usepackage{pdfsync}

\makeatletter
\def\@splitop#1#2\@nil{$\mathscr{#1}\!\!$\calligra#2\,\,}
\newcommand*\DeclareCursiveOperator[2]{%

  \newcommand#1{\mathop{\mbox{\@splitop#2\@nil}}\nolimits}}
\makeatother
\DeclareCursiveOperator{\Bnew}{B}
\DeclareCursiveOperator{\Cnew}{C}
\DeclareCursiveOperator{\Dnew}{D}
\DeclareCursiveOperator{\defe}{Def}

\theoremstyle{plain}
\newtheorem{theorem}{Theorem}[section]
\newtheorem{lemma}[theorem]{Lemma}
\newtheorem{proposition}[theorem]{Proposition}

\newtheorem{remark}[theorem]{Remark}

\theoremstyle{definition}
\numberwithin{equation}{section}

\newcommand{\sr}{\mathcal R}

\newcommand{\di}{\mathrm{dist}\ }
\newcommand{\sub}{\subset}
\newcommand{\B}{\mathcal{B}}
\newcommand{\df}{{\mathrm{def}}}

\newcommand{\hs}{{\mathcal H}}

\newcommand{\cs}{{\mathcal C}}
\newcommand{\ds}{{\mathcal D}}

\newcommand{\bs}{{\mathcal B}}

\newcommand{\E}{{\mathcal E}}

\newcommand{\W}{W}

\newcommand{\R}{{\mathbb R}}
\newcommand{\N}{{\mathbb N}}
\newcommand{\Z}{{\mathbb Z}}

\newcommand{\Om}{\Omega}

\newcommand{\ud}{\,\mathrm{d} }
\newcommand{\flap}{{\mathrm{flat}}}

\newcommand{\weakstar}{\stackrel{\star}{\rightharpoonup}}

\newcommand{\flacon}{\stackrel{\mathrm{flat}}{\to}}

\newcommand{\e}{\varepsilon}
\newcommand{\ep}{\varepsilon}
\newcommand{\en}{{\varepsilon_n}}

\newcommand{\fla}{||\cdot||_{\mathrm{flat}}}

\newcommand{\supp}{\mathrm{supp}\,}
\newcommand{\f}{\varphi}

\newcommand{\curl}{\text{curl }}

\newcommand{\rad}{\mathcal Rad}

\newcommand{\newatop}{\genfrac{}{}{0pt}{1}}

\newcommand{\Huno}{\mathcal{H}^1}

\newcommand{\AS}{\mathcal{AS}}

\newcommand{\osc}{\mathrm{osc}}

\newcommand{\res}{\mathop{\hbox{\vrule height 7pt width .5pt depth 0pt
\vrule height .5pt width 6pt depth 0pt}}\nolimits}

\title
[Low energy configurations of topological singularities]{Low energy configurations of topological singularities in two dimensions: A $\Gamma$-convergence analysis of dipoles} 

\author[L. De Luca]
{Lucia De Luca}
\address[L. De Luca]{SISSA, Via Bonomea 265, 34136 Trieste , Italy}
\email[L. De Luca]{ldeluca@sissa.it}

\author[M. Ponsiglione]
{Marcello Ponsiglione}
\address[Marcello Ponsiglione]{Dipartimento di Matematica ``G. Castelnuovo",  Sapienza Universit\`a di Roma,
Piazzale A. Moro 2, 00185 Roma, Italy} \email[M. Ponsiglione]{ponsigli@mat.uniroma1.it}

\begin{document}
\vskip .2truecm

\begin{abstract}
\small{
This paper deals with the variational analysis of topological singularities in two dimensions. We consider two canonical zero-temperature models: the {core radius approach} and the Ginzburg-Landau energy.  
Denoting by $\ep$ the length scale parameter in such models, we focus on the $|\log\ep|$ energy regime.
It is well known that, for configurations whose energy is bounded by $c |\log \e|$, 
the vorticity measures can be decoupled into the sum of  a finite number of Dirac masses, each one of them carrying $\pi |\log \e|$ energy, 
plus a measure supported on small zero-average sets.
Loosely speaking, on such sets the vorticity measure is close, with respect to  the flat norm, to  zero-average clusters of positive and negative masses. 

Here we perform a compactness and $\Gamma$-convergence analysis  accounting  also for the presence  of such  clusters of dipoles (on the range scale $\ep^s$, for $0<s<1$), which vanish in the flat convergence and whose energy contribution has, so far, been neglected. Our results refine and contain as a particular case the classical 
$\Gamma$-convergence analysis for vortices,  extending it also to low energy configurations consisting of just clusters of dipoles, and  whose energy is of order $c |\log \e|$ with $c<\pi$.

\vskip .3truecm \noindent Keywords: Ginzburg-Landau Model, Topological Singularities, Calculus of Variations.
\vskip.1truecm \noindent 2000 Mathematics Subject Classification: 35Q56, 58K45, 49J45, 35J20.
}
\end{abstract}
\maketitle


\section*{Introduction}
Beyond its relevant applications in  Physics and Materials Science, the analysis of topological singularities, as the length-scale parameter $\e$ tends to zero, is a very fascinating problem in mathematical analysis. 
A celebrated model for the study of topological singularities is the so-called Ginzburg-Landau functional. We deal with its very basic version, i.e., without magnetic field. Let $\Omega\subset\R^2$ be open, bounded, with smooth boundary. 
For any $\e>0$, the 
Ginzburg-Landau functionals $GL_\e:H^1(\Om;\R^2)\to [0,+\infty]$
are defined as
\begin{equation}\label{GL}
GL_\e(u):= \int_{\Om} \frac{1}{2} |\nabla u|^2 +
\frac{1}{\e^2} \W(|u|) \ud x,
\end{equation}
where $W\in C^0([0,+\infty))$ is such that
$W(t)\geq 0$, $W^{-1}\{0\}=\{1\}$ and
$$
\liminf\limits_{t\to 1} \frac{W(t)}{(1-t)^2}>0,\quad\liminf\limits_{t\to \infty} W(t) >0\,.
$$
In the monography \cite{BBH}, Bethuel, Brezis and H\'elein collect the main results about the asymptotic behaviour (as $\ep\to 0$) of the minimizers of $GL_\ep$ with a prescribed boundary datum $g:\partial \Omega\to \mathcal{S}^1$ having non-zero degree. Since then, much work has been devoted to understand the behaviour of sequences of functions $\{u_\ep\}$ which are not necessarily minimizers but satisfy prescribed energy bounds; the natural language to face this problem is provided by the notion of $\Gamma$-convergence.  

The starting point of such analysis has been the study of the regime $|\log\ep|$, corresponding to a  finite number of singularities  in the limit.
 Sharp lower bounds for the energy $\frac{GL_\ep}{|\log\ep|}$ are given in \cite{S,jerrard}. In \cite{JS} a $\Gamma$-convergence result in $W^{1,1}(\Omega;\R^2)$ is provided together with  a compactness analysis of 
 the vorticity measures, identified with the Jacobians $Ju_\ep$ of $u_\ep$. Specifically, up to  a subsequence, 
 the Jacobians
 $Ju_\ep$ converge  in the dual norm of H\"older continuous functions to a measure consisting of a finite sum of Dirac masses, representing the limit vortices. Self-contained short proofs of the compactness of the Jacobians in the flat norm and of the $\Gamma$-
convergence result are collected  in \cite{AP}.   
The ``first-order'' $\Gamma$-convergence of the functional $GL_\ep-M\pi|\log\ep|$ (where $M$ is the number of the singularities) has been largely studied (see, for instance, \cite{San_ser, SS0,AP}). The $\Gamma$-limit is  the  so-called {\it renormalized energy}, depending on the position of the limit singularities and governing their dynamics \cite{SS2, Se, FX, CJ}.
We end up with this list by recalling that  in \cite{ABO} the $\Gamma$-convergence analysis of $GL_\ep$ is developed in any dimension and codimension.

Another natural and perhaps simpler model for  topological singularities, particularly popular in Materials Science,  is the {\it core radius approach}.  Here the order parameter takes  values in $\mathcal{S}^1$ and has a finite number of singularities. The core radius approach consists in drilling the domain, by removing disks of radius $\e$ around each singularity, cutting in this way the logarithmic tail of the energy.    
In the specific model we deal with, we consider convenient to enforce that the singularities have minimal mutual distance of order $\e$:  The set of admissible configurations of topological singularities is    defined as
$$
 X_\ep(\Om):=\{\mu=\sum_{i=1}^{N} z_i \delta_{x_i}\,:\, N\in\N,\,z_i\in \Z,\,x_i\in\Omega,\, \di(x_i,\partial\Om) \ge 2\e, \, |x_i-x_j|\ge 4\e \quad \forall i\neq j \}\,.
$$
The energy functional $\E_{\ep}:X_\e(\Om) \to\R$ induced by the distribution of singularities $\mu$ is given by
\begin{equation}\label{cra}
\E_{\ep}(\mu):= \frac{1}{2}\min_{v\in\AS_{\ep}(\mu)} \int_{\Om_\e(\mu)} |\nabla v|^2  \ud x\,,
\end{equation}
where
$$
\Om_{\ep}(\mu):=\Om\setminus\bigcup_{x_i\in\supp \mu}B_{\ep}(x_i)
$$
and the class of admissible order parameters  associated to $\mu$ is given by
$$
\AS_{\ep}(\mu):=\displaystyle \Big\{ v \in H^1(\Om_\ep(\mu);\mathcal S^1):\,\deg(v,\partial B_\e(x_i)) = \mu(x_i) \, \text{ for all } x_i\in \supp\mu \Big\}.
$$
The $\Gamma$-convergence analysis for the functionals $\E_\ep$ gives the same outcome of the one developed for $GL_\ep$; in particular, a sequence of measures $\{\mu_\ep\}$ with $\sup_{\ep}\frac{\E_\ep(\mu_\ep)}{|\log\ep|}<+\infty$ converges, up to a subsequence, to a finite sum  $\mu$ of Dirac deltas (see \cite{P}). 

In the results mentioned so far,  the $\Gamma$-convergence analysis is done  with respect to the flat norm of the vorticity measures,  neglecting somehow the contribution of  short (in terms of $\ep$) clusters of dipoles of vortices.
In this respect a natural question is how to describe these clusters, quantifying their energy contribution. A first result in this direction has been proven in \cite{JSp} for the Ginzburg-Landau functional.  In such a paper, the authors provide some fine estimates on the flat distance between  $Ju_\ep$ and the class of measures which are sum of Dirac masses with integer coefficients.  In view of this result  we can look at $Ju_\ep$ as a superposition
of zero-average clusters and isolated vortices. 

In this paper, we analyze the behaviour and the energy contribution of zero-average  clusters  whose size can be expressed in terms of $\ep^s$, with $0<s<1$. To this end, we first consider the measures $\mu_\ep$ (resp. $Ju_\ep$) and their convolution $\mu_\ep^s$ (resp. $J^s u_\ep$) with a mollifier whose support is of order $\ep^s$.
It turns out that these measures average out all clusters of dipoles whose size is smaller than $\e^s$, while they provide a good description  for the effective vorticity at mesoscopic scales of order $\e^s$.  
We show that, in the $|\log\ep|$ regime, $\mu^s_\ep$ (resp. $J^s u_\ep$) weak star converge, up to a subsequence, to the flat limit $\mu$ of $\mu_\ep$ (resp. $Ju_\ep$). In this respect we obtain a new compactness property, namely in the  weak star topology,  for the vorticity measures.

In order to account for the short dipoles, we consider the measures $|\mu_\ep^s|$ (resp. $|J^s u_\ep|$), and we prove that they  also enjoy suitable weak star compactness properties. 
The advantage of considering these  total variation measures is that their limits keep track of all clusters of vortices of size larger than $\e^s$.  In particular, the limit family of measures parametrized by $s$  classifies them according to their length.  
 Specifically, we prove that, up to a subsequence independent of $s$, the measures $|\mu_\ep^s|$ (resp. $|J^s u_\ep|$)  converge to some  limit measure $\nu^s$, with $\nu^s=|\mu|+2\xi^s_\df$, being $\xi^s_\df$ a positive sum of  Dirac deltas. Moreover, $\nu^s$ is piecewise constant and non-decreasing with respect to $s$ and  it has a countable set $S=\{0=s_0<s_1<\ldots\}$ of jumps. The measures $\xi^s_\df$  describe the limit density of zero-average clusters of vortices at all the scales  parametrized in terms of all powers $s\in(0,1)$ of $\e$.   
Finally, we prove that the $\Gamma$-limit of the energy $\frac{\E_\ep}{|\log\ep|}$ (resp. $\frac{GL_\ep}{|\log\ep|}$) with respect to the convergence of $\mu_\ep$, $\mu_\ep^s$ and $|\mu_\ep^s|$ (resp. $Ju_\ep$, $J^s u_\ep$ and $|J^su_\ep|$) is given by
$$
\pi\sum_{k=1}^{\sharp S}(s_k-s_{k-1})\nu^{s_{k-1}}(\Om)\,.
$$ 
In particular, this energy is minimized for $\xi^s_\df\equiv 0$, corresponding to $\nu^s\equiv |\mu|$, and in this case it gives back the classical $\Gamma$-limit $\pi |\mu|(\Om)$. 
Moreover, our $\Gamma$-convergence analysis provides a new non-trivial outcome whenever the energy bound is  lower than $\pi |\log \e|$, that is, lower than the minimal energy of a single isolated vortex. More precisely, for configurations with energy of order $c |\log \e|$, with $c<\pi$,  the macroscopic limit  vorticity is zero, while $\xi_\df^s$, and in turn $\nu^s$ can be different than zero for all $0<s<1$ large enough (depending on the prefactor $c$).

To conclude, let us mention that clusters of vortices with zero-average are relevant in many physical systems. They first appear (and then remain together with isolated vortices) as the temperature, and more in general the free energy, increases. In the context of screw dislocations in crystals, they are somehow identified with the so-called statistically stored dislocations. 
Our analysis is a first attempt to describe these objects quantifying their energy contribution, within a purely variational approach at zero temperature, in the rigorous framework of $\Gamma$-convergence.    
We believe that the analysis we have developed here for $GL_\ep$ and $\E_\ep$ can be extended with minor variations to the case of discrete vortices in the $XY$ model and screw dislocations in crystal plasticity \cite{ADGP,AC,P}, whereas an extension to  semi-discrete models for edge dislocations \cite{DGP} appears less clear, and in our opinion deserves future investigations.

\section{Notations and preliminary results}
In this section we introduce the notations that we will use throughout the paper.
We start by fixing    an open bounded subset $\Omega$ of $\R^2$ with Lipschitz continuous boundary.
\subsection{Weak star and flat convergence}
Let $C_c(\Omega)$ be the space of continuous functions compactly supported in $\Omega$
endowed with the $L^\infty$ norm.
A sequence $\{\mu_n\}$ of measures {\it weak star converges} in $\Omega$ to a measure $\mu$ if  for any $\f\in C_c(\Omega)$
$$
\langle\mu_n,\f\rangle\to\langle\mu,\f\rangle\qquad\textrm{as }n\to+\infty\,. 
$$ 
In the following, wherever it is not specified, $\weakstar$ will denote the weak star convergence in $\Omega$.
Moreover, let
 $C^{0,1}(\Om)$ be the space of Lipschitz
continuous functions on $\Om$  endowed with the norm 
$$
\|\psi\|_{C^{0,1}} := \sup_{x\in\Om} |\psi(x)|+ \sup_{\newatop{x,y\in\Om}{  x\neq y}} \frac{|\psi(x) - \psi(y)|}{|x-y|},
$$
and let $C^{0,1}_c(\Om)$ be its subspace of functions  with compact support. 
The   norm in the dual of $C^{0,1}_c(\Om)$ will be denoted by $\fla$ and referred to as {\it{flat norm}},  while $\flacon$ denotes the convergence with respect to the flat norm.
\subsection{Jacobian, current and degree} 
Given $u\in H^1(\Om;\R^2)$, the Jacobian $J u$ of $u$ is the $L^1$
function defined by
$$
J u:= \text{det} \nabla u.
$$
For every $u\in H^1 (\Om;\R^2)$, we can consider $J u$ as an
element of the dual of $C^{0,1}(\Om)$ by setting
$$
\langle J u , \psi\rangle := \int_{\Om} J u  \, \psi \, \ud x \qquad \text{ for any }
\psi\in C^{0,1}({\Om}).
$$
Notice that $Ju$ can be written in a
divergence form as
$J u= \text{div } (u_1 (u_2)_{x_2}, - u_1 (u_2)_{x_1})$,
i.e., for any $\psi \in C^{0,1}_c(\Om)$,
\begin{equation}\label{jwds}
\langle J u, \psi\rangle= - \int_{\Om} u_1 (u_2)_{x_2} \psi_{x_1} - u_1
(u_2)_{x_1} \psi_{x_2} \ud x.
\end{equation}
Equivalently,   we have $J u = \curl (u_1 \nabla u_2)$ and $J u = \frac{1}{2} \curl j(u)$, 
where 
$$
j(u):= u_1 \nabla u_2 - u_2 \nabla u_1
$$ 
is the so-called {\it current}. 

Let $A\subset \Om$ be open with Lipschitz boundary, and let $h\in H^{\frac12}(\partial A; \R^2)$ with $|h|\ge \alpha >0$. The {\it degree} of $h$ is defined as follows
\begin{equation*}
\deg(h,\partial A):= \frac{1}{2\pi} \int_{\partial A} j(h/|h|) \cdot \tau\ud\Huno,
\end{equation*}
where $\tau$ is the tangent field to $\partial A$. 
In \cite{Bo,BN} it is proven that the definition above is well-posed, 
it is  stable with respect to the strong convergence in $H^\frac12(\partial A;\R^2\setminus B_\alpha)$ and 
that  $\deg(h,\partial A)\in\Z$ (here $B_\alpha= B_\alpha(0)$ stands for the ball of radius $\alpha$ centered at zero).
Moreover, if $u\in H^1(A;\R^2\setminus B_\alpha)$ for some $\alpha>0$, then $\deg(u,\partial A) = 0$ 
(here and in the following  we identify $u$ with its trace). 
Finally, if $|u|=1$ on $\partial A$, by Stokes theorem (and by approximating $u$ with smooth functions) we deduce 
\begin{equation}\label{defcur}
\int_A Ju \ud x= \frac{1}{2} \int_A \curl j(u)\ud x := \frac{1}{2} \int_{\partial A} j(u) \cdot \tau \ud\Huno= \deg(u,\partial A).
\end{equation}

Notice that any $u\in H^1(A;\R^2\setminus B_\alpha)$ can be written in polar coordinates as $u(x) = \rho(x) e^{i \theta(x)}$ on $\partial A$ with $|\rho|\ge\alpha$, where $\theta$ is the so called {\it lifting }of $u$. 
By \cite[Theorem 1]{BB} (see also \cite[Remark 3]{BB}), if $A$ is simply connected and $\deg(u,\partial A)=0$, then the lifting   can be  selected in $H^{\frac12} (\partial A)$ with the map $u\mapsto \theta$ continuous (but the image of a bounded subset of  $H^{\frac12} (\partial A;S^1)$ is not necessarily  bounded in $H^{\frac12} (\partial A)$).  
If the degree $d$ is not zero, then  the lifting 
  can be locally selected in $H^{\frac12} (\partial A)$ with a ``jump'' of order $2\pi d$. 
\par
Let us introduce a  notion of modified Jacobian (a variant of the notion introduced  in \cite{ABO}),  which  we will use in our $\Gamma$-convergence results. 
Given $0<\tau<1$ and $u\in H^1(\Om;\R^2)$, set
\begin{equation}\label{moja}
u_{\tau}:= T_\tau(|u|) \frac{u}{|u|}, \qquad J_\tau u: = J u_\tau, \qquad \mbox{ where } T_\tau(\rho)= \min\left\{\frac{\rho}{\tau}, 1\right\}.
\end{equation}

Notice that, for every $v:=(v_1, v_2), \,
w:=(w_1, w_2)$ belonging  to $H^{1}(\Om;\R^2)$
 we have
\begin{equation}\label{for}
J v - J w = \frac{1}{2} \big(J (v_1 - w_1, v_2 + w_2) - J
(v_2-w_2,v_1+w_1)\big).
\end{equation}
By \eqref{jwds} and \eqref{for} we immediately deduce the
following lemma.
\begin{lemma}\label{cosu}
There exists a universal constant $C>0$ such that for any $v,w\in H^1(\Om;\R^2)$, there holds
$$
\|J v - J w\|_\flap\le C\|v - w\|_{L^2} (\|\nabla v\|_2 + \|\nabla w\|_{L^2})\,.
$$ 
\end{lemma}
By Lemma \ref{cosu} we easily obtain  the following proposition.
\begin{proposition}\label{sempre}
Let $\{u_n\}$ be a sequence in $H^1(\Om;\R^2)$ such that $GL_\en(u_n) \le C|\log\en|$, and let $\delta\in(0,\frac 12)$.
Then there exists $C_\delta>0$ such that
\begin{align*}
& \sup_{ \tau\in(\delta,1-\delta)}\| J u_n - J_{\tau} u_n\|_{\flap}\le  C_\delta\en|\log\en|, \\
& \sup_{ \tau\in(\delta,1-\delta)}|J_{\tau}u_n|(\Omega)\le C_\delta |\log\en|\,.
\end{align*}
\end{proposition}  

\subsection{Mollifiers} We denote by $\rho$ a mollifier in $\R^2$, i.e., a positive, $C^\infty$ and  radially symmetric scalar function  compactly supported in $B_1(0)$ with $\int_{\R^2}\rho(x)\ud x=1$. Moreover, for any $\eta>0$, we define $\rho_\eta(\cdot):=\frac{1}{\eta^2}\rho(\frac{\cdot}{\eta})$. We recall that $\rho_\eta\in C^\infty$, $\supp \rho_\eta\sub B_\eta(0)$ and $\int_{\R^2}\rho_\eta(x)\ud x=1$.  Finally, for any function $f\in L^1$, we define the mollification of $f$ as
$$
f\ast\rho_\eta(x):=\int_{\R^2}f(y)\rho_\eta(x-y)\ud y\,;
$$ 
analogously, the mollification of a Radon measure $\mu$ is defined via duality by
$$
\langle  \mu\ast\rho_\eta, \f\rangle :=
\langle  \mu, \f \ast\rho_\eta \rangle
\,,\qquad\textrm{for any }\f\in C_c.
$$ 

By the standard properties of mollifiers, for any $\f\in C_c$, we have
\begin{equation}\label{standmol}
\|\nabla(\f\ast\rho_\eta)\|_{L^\infty}\le \|\f\|_{L^\infty}\|\nabla\rho\|_{L^1} \eta^{-1}\,.
\end{equation}

\section{Ball construction}\label{bc:section}
In this section we   revisit the celebrated {\it ball construction}, a  useful machinery for providing lower bounds of the Dirichlet energy in presence of topological singularities. We follow the approach by Sandier \cite{S} (see also \cite{jerrard0, jerrard, AP}). 
\par
Let $\mathcal B=\{B_{r_1}(x_1), \ldots,B_{r_N}(x_N)\}$ be a finite family of open balls in $\R^2$ with $\bar B_{r_i}(x_i) \cap \bar B_{r_j}(x_j) = \emptyset$ for $i\neq j$, and
let  $\mu=\sum_{i=1}^N z_i\delta_{x_i}$ with $z_i\in\Z\setminus\{0\}$\,.

Let moreover $F(\B,\mu,\cdot)$ be a function defined on open subsets of $\R^2$ satisfying the following properties:
\begin{itemize} 
\item[(i)] $F(\B,\mu,A\cup B)\ge F(\B,\mu,A)+F(\B,\mu,B)$ for all $A,\, B$ open disjoint subsets of $\R^2$;
\item[(ii)] for any annulus $A_{r,R}(x)=B_{R}(x)\setminus \bar{B}_r(x)$ with $A_{r,R}(x)\cap\bigcup_{i}\bar B_{r_i}(x_i)=\emptyset$, there holds
\begin{equation}\label{lbastratto}
F(\B,\mu,A_{r,R}(x))\ge\pi|\mu(B_r(x))|\log\frac{R}{r}\,.
\end{equation}
\end{itemize}
\begin{proposition}\label{ballconstr}
There exists a one-parameter family of open balls $\B(t)$ with $t\ge 0$  such that, setting $U(t):=\bigcup_{B\in\mathcal{B}(t)} B$, the following properties hold true:
\begin{enumerate}
\item $\B(0)=\B \, $;
\item $ U(t_1)\subset U(t_2)$  for any $0\le t_1<t_2 \, $;
\item the balls in $\mathcal{B}(t)$ are pairwise disjoint;
\item for any $0\le t_1<t_2$ and for any open set $U \subseteq \R^2 \,$,
\begin{equation}\label{sti3}
F(\B,\mu,U\cap(U(t_2)\setminus U(t_1)))\ge\pi\sum_{\newatop{B\in\B(t_2)}{B\subseteq U}}|\mu(B)|\log\frac{1+t_2}{1+t_1} \, ;
\end{equation} 
\item  $\displaystyle \sum_{B\in \mathcal{B}(t)}r(B)\le(1+t)\sum_{i}r_i$, where $r(B)$ denotes the radius of the ball $B \, .$
\end{enumerate}
\end{proposition}

\begin{proof}
In order to construct the family $\mathcal{B}(t)$,  we closely follow the  strategy of the ball construction due to Sandier and Jerrard.
The ball construction consists in letting the balls alternatively expand and merge into each other as follows. The expansion phase  consists in letting the balls expand, without changing their centers, in such a way that, at each (artificial) time $t$ the radius $r_i(t)$ of the ball centered at $x_i$ satisfies
\begin{equation}\label{timevel0}
\frac{r_i(t)}{r_i} = 1+t \qquad \text{ for all } i.
\end{equation}

The first expansion phase stops at the first time $T_1$ when two balls bump into each other. Then the merging phase begins. It consists in identifying a suitable partition $\{S^1_j\}_{j=1,\ldots, N_n}$ of the family $\left\{B_{r_i(T_1)}(x_i)\right\}$, and, for each subclass $S^1_j$, in  finding a ball $B_{r^1_j}(x^1_j)$ which contains all the balls in $S^1_j$ such that the following properties hold: 
\begin{itemize} 
\item[i)] $B_{r^1_j}(x^1_j)\cap B_{r^1_l}(x^1_l)=\emptyset$ for all $j\neq l$;
\item[ii)] $r^1_j\ \le \sum_{B\in S^1_j} r(B)$.
\end{itemize} 

After the merging, another expansion phase begins:  we let the balls $\left\{B_{r^1_j}(x^1_j)\right\}$  expand in such a way that, for $t\geq T_1$, for every $j$ we have 
\begin{equation}\label{timevel}
\frac{r^1_j(t)}{r^1_j}=\frac{1+t}{1+T_1}\,.
\end{equation}
Again note that $r^1_j(T_1)=r^1_j$. We iterate this procedure thus obtaining a set of merging times $\left\{T_1,\ldots,T_K\right\}$ with $K\le N$ and  a family $\mathcal B(t)$ for all $t\ge 0$; precisely,  $\mathcal B(t)$ is given by $\{B_{r_j(t)}(x_j)\}_j$ for $t\in [0,T_1)$;
for  $t\in[T_k,T_{k+1})$, $\mathcal B(t)$ can be written as $\{B_{r^k_j(t)}(x_j^k)\}_j$ for all $k=1,\ldots, K-1$, while it consists of a single expanding ball for $t\ge T_k$ . 
By construction, we clearly have properties (1), (2) and (3). Moreover, (5) is an easy consequence of \eqref{timevel0}, \eqref{timevel} and property ii).

It remains to show property (4). We preliminarily note that, by (2),
\begin{equation}\label{monovar}
\sum_{\newatop{B\in \B(\tau_1)}{B\subseteq U}}|\mu(B)|\ge \sum_{\newatop{B\in \B(\tau_2)}{ B\subseteq U }}|\mu(B)|\,\qquad\textrm{for any }0<\tau_1<\tau_2.
\end{equation}
 Let $t_1<\bar t<t_2$. In view of \eqref{monovar}, if we show that (4) holds true for  the pairs $(t_1,\bar t)$ and $(\bar t,t_2)$, then (4) follows also for $t_1$ and $t_2$. Therefore,  we can assume without loss of generality that $T_{k}\notin ]t_1,t_2[$  for any $k=1,\ldots,K$.

Let $t_1<\tau<t_2$ and let $B\in \B(\tau)$. Then there exists a unique ball $B'\in\B(t_1)$ such that $B'\sub B$. By construction, $\mu(B)=\mu(B')$ and by  \eqref{lbastratto} we have 
\begin{equation*}
F(\B,\mu,B\setminus B')\ge \pi|\mu(B)|\log\frac{1+\tau}{1+t_1},
\end{equation*}
which, summing up over all $B\in \B(\tau)$ with $B\subseteq U$, and using \eqref{monovar}, yields
\begin{eqnarray*}
F(\B,\mu, U\cap(U(t_2)\setminus U(t_1)))\ge \pi\sum_{\newatop{B\in \B(\tau)}{ B\subseteq U}}|\mu(B)|\log\frac{1+\tau}{1+t_1}\ge  \pi\sum_{\newatop{B\in \B(t_2)}{ B\subseteq U}}|\mu(B)|\log\frac{1+\tau}{1+t_1} .
\end{eqnarray*}
Property (4) follows by letting $\tau\to t_2$. 
\end{proof}

The following lemma collects some convergence results that will be used in the proofs of our main results.
For any given  $\psi : \, E\subset\R^2 \to\R$ we set
$\osc_{E}(\psi):=\sup_E \psi-\inf_E\psi$. Moreover, for any family $\Bnew$ of balls  we define $\rad (\Bnew):=\sum_{B\in\Bnew}r(B)$. Finally, we often denote by $x_B$ the center of a ball $B$\,.

\begin{lemma}\label{speriamo}
There exists a constant $C>0$ such that the following holds true.
Let $\Bnew$ be a family of pairwise disjoint balls in $\R^2$  and let $\Cnew$ be the family of balls in $\Bnew$ which are contained in $\Omega$.
Let moreover $\alpha, \, \beta$ be two Radon measures supported in $\Omega$ with 
\begin{equation}\label{asslemma}
\supp\alpha\subset\bigcup_{B\in\Cnew}B,\qquad \supp\beta\subset\bigcup_{B\in\Bnew}B\qquad\textrm{and}\quad\alpha(B)=\beta(B)\quad \textrm{for any }B\in\Cnew\,.
\end{equation}
Then, for any $\eta>0$, there holds:
\begin{itemize}
\item[(i)] $\|\alpha-\beta\|_\flap\le C\,\rad (\Bnew)(|\alpha|+|\beta|)(\Omega)$\,;
\item[(ii)] $\sum_{B\in\Cnew}  |((\alpha-\beta)\res B)\ast\rho_{\eta}| \le C   \|\nabla\rho\|_{L^1} \, \eta^{-1}\rad (\Bnew)(|\alpha|+|\beta|)(\Omega) $\,;
\item[(iii)] for any $\f\in C_c(\Omega)$ 
$$
\sum_{B\in\Cnew}|\langle |\alpha(B)|\delta_{x_B}- |\alpha(B)|\delta_{x_B}\ast\rho_{\eta},\f\rangle|\le C\,|\alpha|(\Omega)\,\omega_{\f}(\eta),
$$
where $\omega_\f$ denotes the modulus of continuity of $\f$.
\end{itemize}
\end{lemma}
\begin{proof}
We divide the proof in three steps corresponding to the three facts stated in the Lemma. 
\quad\\
{\it Step 1: Proof of (i).} Set $\Dnew:=\Bnew\setminus\Cnew$. Let $\psi\in C_{c}^{0,1}(\Omega)$ with $\|\psi\|_{C^{0,1}}\le 1$. By  \eqref{asslemma} we have
\begin{multline*}
\langle\alpha - \beta,\psi\rangle = \sum_{B\in\Cnew} \int_{B} \psi \ud(\alpha - \beta) - \sum_{B\in\Dnew}\int_{B} \psi\ud \beta
\\
\le\sum_{B\in \Cnew}  \osc_{B}(\psi)  \, (|\alpha| + |\beta|)(B) + \sum_{B\in\Dnew} \max_{B} |\psi|  \,|\beta|(B)
\\ 
\le\sum_{B\in  \Cnew}\mbox{diam}(B)\,(|\alpha| + |\beta|)(B)+\sum_{B\in
\Dnew}\mbox{diam}(B)\, |\beta|(B)
\\
 \le\sum_{B\in\Bnew} \mbox{diam}(B)\, (|\alpha| + |\beta|) (\Omega)=2\,\rad(\Bnew) (|\alpha| + |\beta|) (\Omega).
\end{multline*}
By taking the $\sup$ over all $\psi$ we get (i).

{\it Step 2: Proof of (ii).} Let $B\in\Cnew$ and let $\f\in C_c(B)$.
By \eqref{asslemma} and \eqref{standmol}, we have
\begin{multline*}
 \langle ((\alpha-\beta)\res B)\ast\rho_{\eta} ,\f\rangle 
 =  \langle (\alpha-\beta)\res B,\f \ast\rho_{\eta}\rangle \\
\le  \osc_{B}(\f\ast\rho_{\eta})(|\alpha|+|\beta|)(B) 
 \le \|\nabla(\f \ast\rho_{\eta})\|_{L^\infty}\, \mbox{diam}(B)   \,(|\alpha|+|\beta|)(\Omega)\\
   \le  \|\f\|_{L^\infty}\, \|\nabla\rho\|_{L^1} \, \eta^{-1}\, \mbox{diam}(B) \,(|\alpha|+|\beta|)(\Omega)\,.
\end{multline*}
By taking the $\sup$ over all $\f$ and summing over all $B\in\Cnew$, we get (ii).


{\it Step 3: Proof of (iii).} 
Let $\f\in  C_c(\Omega)$. Then
\begin{multline*}
\sum_{{B\in \Cnew}}  
|\langle  |\alpha(B)| \delta_{x_B} - |\alpha(B)| \delta_{x_B} \ast \rho_{\eta}, \f \rangle|= \sum_{{B\in \Cnew}}   |\alpha(B)|
|\langle   \delta_{x_B} , \f - \f \ast \rho_{\eta}\rangle| \\
\le\sum_{{B\in \Cnew}}   |\alpha(B)|\,\omega_\f(\eta)=|\alpha|(\Omega)\, \omega_\f(\eta)\,.
\end{multline*}
This concludes the proof of (iii) and of the lemma.

\end{proof}
We set
$$
X(\Om):=\{\mu=\sum_{i=1}^{N} z_i \delta_{x_i}\,: \, N\in\N, \, z_ i\in\Z, \, x_i\in\Om\}.
$$
Moreover, for any countable set $S\subset\R$ we denote by $\sharp S$ the cardinality of $S$. If $S$ is infinitely countable, with a little abuse of notations,  we write ``$k=1,\ldots,\sharp S$'' in place of ``$k\in\N$''.
Finally, here and throughout the whole paper, $C$ denotes a positive universal constant which may change from line to line. We write $C_a$ whenever we want to stress the dependence of $C$ on some parameter $a$.

\begin{theorem}\label{mains1}
Let $\mathcal B_n = \{B_{r_{i,n}}(x_{i,n})\}$ be  a sequence of finite families  of disjoint balls in $\R^2$ with ${\sr}_n:=\rad(\mathcal B_n)\to 0$ as $n\to +\infty$, and let $\mu_n:=\sum_{i} z_{i,n} \delta_{x_{i,n}}$, with $z_{i,n}\in\Z$. Assume that
\begin{equation}\label{enebound}
F(\mathcal B_n, \mu_n,\Om) +|\mu_n|(\Omega)\le C |\log \sr_n|,
\end{equation}
for some  constant $C>0$ independent of $n$. Set $\mu_n^s:=\mu_n\ast\rho_{\sr_n^s}$ for any $0<s<1$. 

There exist  $\mu\in X(\Om)$,
a countable (finite or infinite) set $S:=\{0=s_0<s_1<s_2<\ldots\}$ with $\sup S=1$, and 
a family   $\{\nu^s\}_{s\in [0,1)}\subset X(\Om)$  such that the following facts hold true.
\begin{itemize}
\item[(i)] {Flat and weak-$\star$ compactness for vorticity measures}: Up to a subsequence $\mu_n\flacon \mu$ and
for any $s\in (0,1)$,  $\mu_n^s \weakstar \mu$ up to a  subsequence (independent of $s$).
\item[(ii)] {Weak-$\star$ compactness for vorticity densities}: For any $s\in (0,1)\setminus S$, $|\mu_n^s|\weakstar \nu^s$  up to a  subsequence (independent of $s$).
\item[(iii)] {Structure of vorticity densities}: 
For all $s\in[0,1)$,  $\nu^s= |\mu| + 2 \xi_{\mathrm{def}}^s$ for some {\it defect measure}  $\xi^s_{\mathrm{def}}\in X(\Om)$ with $\xi^s_{\mathrm{def}}\ge 0$. 
Moreover, $\nu^s$  is constant in $[s_{k-1},s_{k})$ with
$\nu^{s_{k-1}} \le \nu^{s_{k}}$ for any $k=1,\ldots,\sharp S$.
\item[(iv)]{Lower bound}: $\displaystyle \liminf_{n\to +\infty} \frac{F(\mathcal B_n, \mu_n,\Om)}{|\log\sr_n|}\ge\pi\sum_{k=1}^{\sharp S}(s_k-s_{k-1})\nu^{s_{k-1}}(\Omega)$.
\end{itemize}
\end{theorem}

\begin{proof}
Let $\mathcal B_n(t)$, for any $n\in \N$, be a time parametrized family of balls, starting from $\mathcal B_n$,  as in Proposition \ref{ballconstr}. 
Set 
$
\cs_n(t):= \{ B\in\bs_n(t), \, B\subset \Om\}, \quad \ds_n(t):=
\mathcal B_n(t)\setminus \cs_n(t), \quad U_n(t):=\bigcup_{B\in\bs_n(t)}B\,.
$

Moreover, for any $0<s<1$ set 
\begin{equation*}
t_n^s:=\frac{1}{\sr_n^{1-s}} - 1, \qquad
\tilde\mu_n^s:= \sum_{B_r(x)\in \cs_n(t_n^s)}  \mu_n(B_r(x)) \delta_x\,. 
\end{equation*}
By the energy bound \eqref{enebound} and by applying \eqref{sti3} with $U=\Omega$, $t_1=0$ and $t_2=t_n^s$,  we have 
\begin{multline*}
 C |\log \sr_n| \ge F(\mathcal B_n, \mu_n,\Om\cap (U_n(t_n^s)\setminus U_n(0)))\\
\ge 
\pi \sum_{B\in \cs_n(t_n^s)}|\mu_n(B)|(1-s)|\log\sr_n|=\pi (1-s)|\tilde\mu_n^s|(\Omega)|\log\sr_n|\ 
\end{multline*}
and hence 
\begin{equation}\label{tvb}
|\tilde\mu_n^s|(\Omega)\le\frac{C}{1-s}. 
\end{equation}
By \eqref{tvb}, 
 up to a  subsequence, $\tilde\mu_n^{s}$ converges to some $\mu^s$, both in the weak star and in the flat sense,  for some $\mu^s\in X(\Omega)$. Let us show that in fact $\mu:=\mu^s$ does not depend on $s$.
 To this purpose,  it is enough to prove that for any $0<\sigma_1<\sigma_2<1$, $\|\tilde\mu_n^{\sigma_1} - \tilde\mu_n^{\sigma_2}\|_{\mathrm{flat}}\to 0$
 as $ n\to +\infty$.
By construction, $(\tilde\mu_n^{\sigma_1} - \tilde\mu_n^{\sigma_2})(B) = 0$ for every $B\in\cs_n(t_n^{\sigma_1})$. 
Therefore, by applying Lemma \ref{speriamo} (i) with $\Bnew=\bs_{n}(t_n^{\sigma_1})$, $\alpha= \tilde\mu_n^{\sigma_1}$ and $\beta= \tilde\mu_n^{\sigma_2}$, we get 
\begin{equation}\label{stimaflat}
\|\tilde\mu_n^{\sigma_1} - \tilde\mu_n^{\sigma_2}\|_{\mathrm{flat}} \le C\, \rad(\bs_n(t_n^{\sigma_1})) (|\tilde\mu_n^{\sigma_1}|(\Omega)+| \tilde\mu_n^{\sigma_2}|(\Omega))\\
\le \frac{C}{1-\sigma_2} \sr_n^{\sigma_1}   \to 0 \quad \mbox{ as } n\to+\infty\,,
\end{equation}
where in the last inequality  we have used \eqref{tvb} and Proposition \ref{ballconstr} (5).  

Moreover, by applying again Lemma \ref{speriamo} (i) and Proposition \ref{ballconstr} (5), and using also the energy bound \eqref{enebound},
 we easily deduce that for every $0<s<1$ 
\begin{equation}\label{stimaflat2}
\| \mu_n - \tilde\mu_n^{s}\|_{\mathrm{flat}} \le 
 C  |\log \sr_n| \sr_n^{s}  
\quad \to 0 \qquad \mbox{ as } n\to+\infty\,.
\end{equation}

Therefore, by \eqref{tvb}, \eqref{stimaflat} and \eqref{stimaflat2}, 
\begin{equation}\label{utile}
\tilde\mu_n^s\weakstar\mu \qquad\textrm{for any }0<s<1\qquad\textrm{and}\qquad \mu_n\flacon\mu
\end{equation}
up to a subsequence independent of $s$.

We now prove the second part of (i).
In view of \eqref{tvb}, it is immediate to see that for any $0<\sigma_1,\sigma_2<1$
$$
\tilde\mu_n^{\sigma_2}\ast\rho_{\sr^{\sigma_1}_n}-\tilde\mu_n^{\sigma_2} \weakstar 0 \qquad \mbox{ as } n\to +\infty\,.
$$
Therefore, in virtue of \eqref{utile}, the claim (i) is proven if we show that for any $0<\sigma_1<\sigma_2<1$
\begin{equation*}
\mu_n^{\sigma_1}-(\tilde\mu_n^{\sigma_2}\ast \rho_{\sr_n^{\sigma_1}}) \weakstar 0  \qquad \mbox{ as } n\to +\infty\,.
\end{equation*}
Let $\f\in C_c(\Omega)$. 
By applying Lemma \ref{speriamo} (ii) with $\Bnew=\bs_n(t_n^{\sigma_2})$, $\alpha=\tilde\mu_n^{\sigma_2}$, $\beta=\mu_n$ and $\eta=\sr_n^{\sigma_1}$,  for $n$ large enough we get
\begin{multline*}
\left| \langle\mu_n^{\sigma_1}-\tilde\mu_n^{\sigma_2}\ast \rho_{\sr_n^{\sigma_1}},\f\rangle \right|
 =  \sum_{B\in\cs_n(t_n^{\sigma_2})} \left| \langle ((\mu_n-\tilde\mu_n^{\sigma_2})\res B)\ast\rho_{\sr_n^{\sigma_1}},\f\rangle \right|\\
 \le  C\|\f\|_{L^\infty}\, \|\nabla\rho\|_{L^1} \, \sr_n^{-\sigma_1}\rad (\bs_n(t_n^{\sigma_2}))(|\mu_n|+|\tilde\mu_n^{\sigma_2}|)(\Omega)\\
 \le  C\|\f\|_{L^\infty}\, \|\nabla\rho\|_{L^1}  \, \sr_n^{\sigma_2-\sigma_1}\,|\log\sr_n|\to 0\,  ,
\end{multline*}
where in the last inequality we have used again Proposition \ref{ballconstr} (5) and the bound \eqref{enebound}. 
We pass to the proof of (ii) and (iii). First we recall that, in view of  \eqref{tvb},  the measures 
$|\tilde \mu_n^s|$ are pre-compact. 
Let now $0<\sigma_1<\sigma_2<1$ and assume that for $i=1,2$
$$
|\tilde \mu^{\sigma_i}_{n_h}|\weakstar \tilde \nu^{\sigma_i} \qquad \text{ as } n_h\to+\infty,
$$
for some subsequence $n_h\to+ \infty$ (independent of $i$) and for some measures $\tilde \nu^{\sigma_i}$, that by construction satisfy 
$$
\frac 12 (\tilde \nu^{\sigma_i} - |\mu|) \in X(\Om)\,.
$$ 
Then, it is easy to see that  $\tilde \nu^{\sigma_1} \le \tilde \nu^{\sigma_2}$ in the sense of measures. Since this happens for every pairs of limit measures $(\tilde\nu^{\sigma_1},\tilde\nu^{\sigma_2})$, arguing as in the proof of the classical Helly's Theorem we deduce that there exists a non-decreasing family of measures $\tilde \nu^{s}$,
with $s\in(0,1)$, such that, up to a subsequence (not relabeled and independent of $s$)
\begin{equation}\label{priva}
|\tilde \mu^{s}_{n}|\weakstar \tilde \nu^{s} \qquad \text{ for all } 0<s<1\,.
\end{equation}
By the monotonicity property, and recalling that 
$\tilde \nu^s$ are finite sum of Dirac masses with positive integer weights, we have that 
 the map $s\mapsto \tilde \nu^s$ is piecewise constant. 

Let $\psi_n: (0,1) \to \N\cup\{0\}$ be the function that to any $s\in (0,1)$ associates the number of balls $B$ in $\mathcal C_n(t_n^s)$ with $\mu_n(B)\neq 0$, and let $S_n$ be the union of the set $\{0,\, 1\}$ with set of discontinuity points of $\psi_n$. 
By \eqref{tvb} the cardinality of $S_n\cap (0,t)$ is uniformly bounded from above by a constant depending only on $t$. Therefore, up to a subsequence, $S_n$ converge in the Hausdorff sense to some discrete set $S\subset [0,1]$  with $0, \, 1\in S$. Moreover, by construction $S$ contains all the discontinuity points of the map $s\mapsto \tilde \nu^s$. 
Let $\nu^{s}$ be the right continuous extension   to the whole interval $[0,1)$ of $\tilde\nu_s$ restricted to $(0,1)\setminus S$. 
By construction $\nu^s$ satisfies all the properties in (iii).

We pass to the proof of (ii). 
In virtue of \eqref{priva}, (ii) follows provided that for any  $s<\sigma$ with $[s,\sigma]\subset (0,1)\setminus S$ there holds
\begin{equation}\label{weakconv}
| \mu^{s}_{n}| - | \tilde \mu^{\sigma}_{n}| \weakstar 0\qquad\textrm{ as }n\to+\infty\,.
\end{equation}
Set
\begin{equation*}
\hat\mu_{n,\neq 0}^{\sigma}:=\sum_{\newatop{B\in \cs_n(t_n^{\sigma})}{\mu_n(B)\neq 0}}\mu_n\res B\quad\textrm{and}\quad \hat\mu_{n,=0}^{\sigma}:=\mu_n-\hat\mu_{n,\neq 0}^{\sigma}.
\end{equation*}
Let $A\subset\subset\Omega$ be open. By applying Lemma \ref{speriamo} (ii) with $\Bnew=\bs_n(t_n^{\sigma})$, $\alpha=0$ and $\beta= \hat\mu_{n,=0}^{\sigma}$, and by using Proposition \ref{ballconstr} (5) and the bound \eqref{enebound}, 
 for $n$ large enough we obtain
\begin{equation*}
|\hat\mu_{n,=0}^{\sigma}\ast \rho_{\sr_n^s}|(A)\le \sum_{\newatop{B\in \cs_n(t_n^{\sigma})}{\mu_n(B)= 0}}|(\mu_n\res B)\ast \rho_{\sr_n^s}|
\le C\|\nabla\rho\|_{L^1}\,|\log\sr_n|\, \sr_n^{\sigma-s}\to 0 \qquad\textrm{ as }n\to+\infty\,.
\end{equation*}
As a consequence, \eqref{weakconv} is equivalent to 
\begin{equation}\label{wcequiv}
|\hat\mu^{\sigma}_{n,\neq 0}\ast \rho_{\sr_n^s}| - | \tilde \mu^{\sigma}_{n}| \weakstar 0\qquad\textrm{ as }n\to+\infty\,.
\end{equation}
Since $[s,\sigma]\subset (0,1)\setminus S$, it is easy to see that,  for $n$ large enough,  the supports of the measures 
$(\mu_n \res B) \ast \rho_{\sr_n^s}$ (for $B \in \cs_n(t_n^{\sigma})$) are pairwise disjoint,  whence
\begin{equation}\label{follia}
|\hat\mu^{\sigma}_{n,\neq 0}\ast \rho_{\sr_n^s}| = \sum_{\newatop{B\in \cs_n(t_n^{\sigma})}{\mu_n(B)\neq 0}}|(\mu_n\res B) \ast \rho_{\sr_n^s}|.
\end{equation}
Let $\f\in C_c(\Om)$ with $\|\f\|_{L^\infty} \le 1$.  By \eqref{follia} and by triangular inequality we have 
\begin{eqnarray}\nonumber
|\langle |\hat\mu^{\sigma}_{n,\neq 0}\ast \rho_{\sr_n^s}| - | \tilde \mu^{\sigma}_{n}|,\f\rangle | \le
\sum_{\newatop{B\in \cs_n(t_n^{\sigma})}{\mu_n(B)\neq 0}} | \langle |(\mu_n\res B) \ast \rho_{\sr_n^s}| -
|\mu_n(B)| \delta_{x_B}, \f \rangle |
\\  \label{0l}
\le \sum_{\newatop{B\in \cs_n(t_n^{\sigma})}{\mu_n(B)\neq 0}} | \langle |(\mu_n\res B) \ast \rho_{\sr_n^s}|
- |\mu_n(B)| \delta_{x_B} \ast \rho_{\sr_n^s}, \f \rangle|\\
\label{1l}
+ \sum_{\newatop{B\in \cs_n(t_n^{\sigma})}{\mu_n(B)\neq 0}} | \langle |\mu_n(B)| \delta_{x_B}\ast \rho_{\sr_n^s}
- |\mu_n(B)| \delta_{x_B}, \f \rangle|\,.
\end{eqnarray}

As for the addendum in \eqref{0l}, we can apply Lemma \ref{speriamo} (ii), Proposition \ref{ballconstr} (5) and \eqref{enebound}, thus obtaining 
\begin{multline*}
\sum_{\newatop{B\in \cs_n(t_n^{\sigma})}{\mu_n(B)\neq 0}} | \langle |(\mu_n\res B) \ast \rho_{\sr_n^s}|
- |\mu_n(B)| \delta_{x_B} \ast \rho_{\sr_n^s}, \f \rangle|\le
\sum_{\newatop{B\in \cs_n(t_n^{\sigma})}{\mu_n(B)\neq 0}}  \langle |(\mu_n\res B) \ast \rho_{\sr_n^s}
- \mu_n(B) \delta_{x_B} \ast \rho_{\sr_n^s}|, |\f| \rangle \\
\le
\sum_{\newatop{B\in \cs_n(t_n^{\sigma})}{\mu_n(B)\neq 0}} |((\mu_n\res B) 
- \mu_n(B) \delta_{x_B})\ast \rho_{\sr_n^s}|
\le C\|\nabla\rho\|_{L^1}|\log\sr_n| \sr_n^{\sigma-s}\to 0\quad\textrm{as }n\to+\infty\,,
\end{multline*}
whereas for the term in \eqref{1l},  by applying Lemma \ref{speriamo} (iii) and recalling \eqref{tvb}, we get
\begin{equation*}
\sum_{\newatop{B\in \cs_n(t_n^{\sigma})}{\mu_n(B)\neq 0}}  
|\langle|\mu_n(B)| \delta_{x_B} \ast \rho_{\sr_n^s}-   |\mu_n(B)| \delta_{x_B} , \f \rangle|
\le C \omega_\f({\sr}_n^s)\to 0\qquad\textrm{ as }n\to+\infty\,.
\end{equation*}
Then, \eqref{wcequiv} follows and (ii) is proven.

Finally, we prove the lower bound (iv). Let $L\in \N$ with $L\le \sharp S$ and let $\eta>0$ (small enough). 
 
By 
\eqref{sti3} we have 
\begin{multline*}
F(\mathcal B_n, \mu_n,\Om) \ge \sum_{l=1}^L  F(\B_n,\mu_n,\Omega\cap(U_n(t_n^{s_{l-1}+\eta})\setminus U_n(t_n^{s_{l}-\eta}))\\
\ge
 \pi|\log \sr_n| \sum_{l=1}^L (s_{l} - s_{l-1} - 2 \eta) \sum_{B\in \cs_n (t_n^{s_{l-1} + \eta})} |\mu_n (B)|
=
 \pi|\log \sr_n| \sum_{l=1}^L (s_{l} - s_{l-1} - 2 \eta)  |\tilde \mu_n^{s_{l-1}+\eta}| (\Om).
\end{multline*}
By \eqref{priva} and recalling that the map $s\to  \nu^s$ is the right-continuous extension of $s\to \tilde \nu^s$, we have
$$
|\tilde \mu_n^{s_{l-1}+\eta}| \weakstar \tilde \nu^{s_{l-1}+\eta} = \nu^{s_{l-1}}.
$$
By the lower semicontinuity property of the total variation with respect to the weak star convergence we get
$$
 \liminf_{n\to +\infty} \frac{F(\mathcal B_n, \mu_n,\Om)}{|\log\sr_n|}\ge\pi\sum_{l=1}^{L}(s_l-s_{l-1} -2\eta)\nu^{s_{l-1}}(\Omega)\,,
$$

from which the lower bound (iv) follows by sending first $\eta \to 0$ and then $L\to \sharp S$.
\end{proof}

For further use, we fix the ``minimal'' functional $F$ satisfying the assumptions (i) and (ii) in Section \ref{bc:section}. First,  if $A_{r,R}(x):=B_R(x) \setminus B_r(x)$ is an annulus  that does not intersect any $B_{r_i}(x_i)$, we set 
$$
G(\mathcal B, \mu, A_{r,R}(x)) :=\pi  |\mu(B_r(x))|  \log\left(\frac{R}{r}\right).
$$ 
Then, for every open set $A\subset\R^2$ we set
\begin{equation}\label{minf}
F(\mathcal B, \mu,A):= \sup \sum_j G(\mathcal B, \mu, A_j),
\end{equation}
where the sup is over all finite families of disjoint annuli $A_j\subset A$ that do not intersect any $B_{r_i}(x_i)$.
Notice that, if $A$ is an annulus that does not intersect any $B_{r_i}(x_i)$,  then $F(\mathcal B, \mu,A)= G(\mathcal B, \mu,A)$.
\begin{remark}\label{disuene}
{\rm
    The definition of $F$ in \eqref{minf} is justified by the following observation. Let $\tilde \Om = \Om\setminus \cup_{B\in\bs} B$. Given $u\in H^1 (\tilde \Om;S^1)$, let $\mu:=\sum_{B\in\cs} \deg(u,\partial B) \delta_{x_B}$, where $\cs$ denotes the family of balls in $\bs$ that are contained in $\Om$, and  $x_B$ is the center of $B$. Then, by Jensen inequality we easily deduce (see for instance \cite{S}) that 
$$
F(\mathcal B, \mu, U)\le \frac12 \int_{U\cap\tilde\Om} |\nabla u|^2\ud x,
$$
for every open set $U\subset \Om$\,.
}
\end{remark}

\section{$\Gamma$-convergence of the core radius approach}
In this section we exploit the results in Section \ref{bc:section} in order to develop a $\Gamma$-convergence analysis for the core radius approach \eqref{cra}. 
\begin{theorem}\label{mainscra}
Let $\{\e_n\}\subset\R^+$ with $\e_n\to 0$ as $n\to +\infty$. The following $\Gamma$-convergence result holds true. 
\begin{itemize}
\item[] (\textbf{Compactness})
Let $\{\mu_n\}$ with $\mu_n \in X_{\e_n}(\Om)$ for all $n\in\N$  be such that 
\begin{equation}\label{eneboundcra}
 \E_{\ep_n}(\mu_n) \le C |\log \e_n| \qquad \text{ for all } n\in \N,
\end{equation}
for some  constant $C>0$ independent of $n$.
Set $\mu_n^s:=\mu_n\ast\rho_{\e_n^s}$ for any $0<s<1$. 

Then, there exist a measure $\mu\in X(\Omega)$,
a countable (finite or infinite) set $S:=\{0=s_0<s_1<s_2<\ldots\}$ with $\sup S=1$, and 
a family  of positive Radon measures $\{\nu^s\}_{s\in [0,1)} \subset X(\Omega)$,  such that  the following compactness and structure properties hold true:
\begin{itemize}
\item[(1)] {Flat and weak star compactness for vorticity measures}: Up to a subsequence, $\mu_n\flacon \mu$ and
for any $s\in (0,1)$,  $\mu_n^s \weakstar \mu$ up to a  subsequence (independent of $s$).
\item[(2)] {Weak star compactness for vorticity densities}: For any $s\in (0,1)\setminus S$, $|\mu_n^s|\weakstar \nu^s$ up to a  subsequence (independent of $s$).
\item[(3)] {Structure of vorticity densities}: 
For all $s\in[0,1)$,  $\nu^s= |\mu| + 2 \xi_{\mathrm{def}}^s$ for some {\it defect measure}  $\xi^s_{\mathrm{def}}\in X(\Om)$ with $\xi^s_{\mathrm{def}}\ge 0$. 
Moreover, $\nu^s$  is constant in $[s_{k-1},s_{k})$ and
$\nu^{s_{k-1}} \le \nu^{s_{k}}$ for any $k=1,\ldots,\sharp S$.
\end{itemize}
\item[] (\textbf{$\Gamma$-liminf inequality}) 
Let  $\{\mu_n\}$ be such that for all $s\in (0,1)$, $|\mu_n^s|\weakstar \nu^s$ for some 
family of measures $\{\nu^s\}_{s\in [0,1)}$ as in (3). Then, 
$$
\displaystyle \liminf_{n\to +\infty} \frac{\E_{\e_n}(\mu_n)}{|\log\e_n|}\ge\pi\sum_{k=1}^{\sharp S}(s_k-s_{k-1})\nu^{s_{k-1}}(\Omega)\,.
$$

\item[] (\textbf{$\Gamma$-limsup inequality}) For any $\mu\in X(\Omega)$ and for any
family of measures $\{\nu^s\}_{s\in [0,1)}$ as in (3) with 
$\sum_{k=1}^{\sharp S}(s_k-s_{k-1})\nu^{s_{k-1}}(\Omega) <+\infty$, there exists 
$\{\mu_n\}$ with $\mu_n \in X_{\e_n}(\Om)$ for all $n\in \N$ such that $\mu_n\flacon\mu$, $\mu^s_n\weakstar \mu$ for any $s\in (0,1)$, $|\mu^s_n|\weakstar \nu^s$ for any $s\in (0,1)\setminus S$, and
\begin{equation}\label{limsup}
\displaystyle \lim_{n\to +\infty} \frac{\E_{\e_n}(\mu_n)}{|\log\e_n|} = \pi\sum_{k=1}^{\sharp S}(s_k-s_{k-1})\nu^{s_{k-1}}(\Omega)\,.
\end{equation}
\end{itemize}
\end{theorem}
\begin{proof}
For any $n\in \N$ we denote by $x_{i,n}$ the points in the support of $\mu_n$ and we set $\bs_n:=\{B_{\en}(x_{i,n})\}$ and $\sr_n:=\rad(\bs_n)$.
Trivially, 
\begin{equation}\label{stimarad}
\sr_n\le \en\, |\mu_n|(\Omega)\,.
\end{equation}

Let $\bs_{n}(t)$ be a time parametrized family of balls given by  Proposition \ref{ballconstr} starting from $\bs_{n}=:\bs_{n}(0)$; we denote by $\cs_{n}(t)$  the family of balls in $\bs_{n}(t)$ that are contained in $\Om$ and by $U_{n}(t)$ the union of the balls in $\bs_{n}(t)$.
Moreover we set
$$
\mu_{n}(t) := \sum_{B \in \cs_{n}(t)}  \deg(u_n, \partial B) \delta_{x_B}\,. 
$$
Let $F$ be as defined in \eqref{minf}. 

\textbf{\emph {Step 1: Proof of compactness.}}
By applying Proposition \ref{ballconstr} (5) with $t_1=0$ and $t_2=1$, we have
\begin{equation}\label{vartotlim}
F(\bs_n,\mu_n,\Omega)\ge \pi\, \log 2\,\sum_{\newatop{B\in\bs_n(1)}{B\subset\Omega}}|\mu_n|(B)=\pi\, \log 2\,|\mu_n|(\Omega)\,,
\end{equation}
where the equality follows by the fact that $\mu_n\in X_{\en}(\Omega)$.
Therefore,  by \eqref{vartotlim},  Remark \ref{disuene} and the energy bound \eqref{eneboundcra}, we obtain
\begin{equation}\label{nuovobounden}
F(\bs_n,\mu_n,\Omega)+|\mu_n|(\Omega)\le C\, F(\bs_n,\mu_n,\Omega)\le C\,\E_{\en}(\mu_n)\le C|\log\en|\,,
\end{equation}
which, in view of \eqref{stimarad}, immediately  implies
\begin{equation*}
F(\bs_n,\mu_n,\Omega)+|\mu_n|(\Omega)\le C|\log\sr_n|\,.
\end{equation*}
By Theorem \ref{mains1} there exist a measure $\mu\in X(\Omega)$, a countable set $S:=\{0<s_1<s_2<\ldots\}$ with $\sup S=1$ and a family of measures $\{\nu^s\}_{s\in [0,1)}\subset X(\Omega)$ satisfying the structure properties in (3), such that up to a subsequence independent of $s$, there holds
\begin{equation}\label{teovecchio1}
\begin{aligned}
&\mu_n\flacon\mu\,, \\ 
&\mu_n\ast\rho_{\sr^s_n}\weakstar\mu \qquad\textrm{for any }s\in (0,1)\,,\\ 
&|\mu_n\ast\rho_{\sr^s_n}|\weakstar\nu^s  \qquad\textrm{for any }s\in (0,1)\setminus S\,.
\end{aligned}
\end{equation}

Therefore, in order to conclude the proof of the compactness it is enough to show that 
\begin{eqnarray}\label{finecomp1}
&\mu_n\ast\rho_{\e^s_n}\weakstar \mu\qquad\textrm{for any }s\in (0,1)\\ \label{finecomp2}
&|\mu_n\ast\rho_{\sr^s_n}|- |\mu_n\ast\rho_{\e^s_n}|\weakstar 0  \qquad\textrm{for any }s\in (0,1)\setminus S\,.
\end{eqnarray}

We preliminarily notice that if $\mu_n\equiv 0$ for $n$ large enough, then  \eqref{finecomp1} and \eqref{finecomp2} are trivially satisfied, so that we can assume without loss of generality that $\mu_n\neq 0$ and hence 
\begin{equation}\label{ipotesifondam}
\sr_n\ge\en\,.
\end{equation}

We start by proving \eqref{finecomp1}. Let $\sigma>0$ and set $t_n^\sigma:=\frac{1}{\e_n^{1-\sigma}}-1$. By construction
\begin{equation*}
(\mu_n - \mu_{n}(t^\sigma_n))(B) = 0\qquad\textrm{ for any }\quad B\in \cs_{n}(t^\sigma_n)\,.
\end{equation*} 
Therefore, by Lemma \ref{speriamo} (i), \eqref{nuovobounden}, Proposition \ref{ballconstr} (5) and \eqref{stimarad}, we have
\begin{multline}\label{comp20}
\|\mu_n - \mu_{n}(t^\sigma_n)\|_{\mathrm{flat}} \le 
 C\,\rad(\bs_{n}(t^\sigma_n))\,(|\mu_n|(\Omega)+|\mu_{n}(t^\sigma_n)|(\Omega))\\
 \le
 C \,\rad( \bs_{n}(t^\sigma_n)) |\log\en| \le C_l |\log\e_n|^{\sigma + 1} \ep^\sigma_n\to 0\quad\textrm{as }n\to +\infty\,,
\end{multline}
which in virtue of \eqref{teovecchio1} yields that, up to subsequences,  
\begin{equation}\label{vecchiolimfla}
\mu_n(t^\sigma_n)\flacon \mu\,.
\end{equation}
Moreover, by \eqref{nuovobounden} and by \eqref{sti3} in Proposition \ref{ballconstr},
 \begin{equation*}
C|\log\en|\ge F(\bs_{n},\mu_{n},\Omega\cap(U_{n}(t^\sigma_n)\setminus U_{n}(0)))\ge \pi|\mu_{n}(t^\sigma_n)|(1-\sigma)|\log\en|\,,
  \end{equation*}
which, together with \eqref{vecchiolimfla},  implies that, up to a subsequence independent of $s$, 
 $\mu_n(t_n^\sigma)\weakstar\mu$ and 
 \begin{equation}\label{convdeb0}
 \mu_n(t_n^\sigma)\ast\rho_{\ep_n^s}\weakstar\mu\,. 
 \end{equation}
 Let now $\sigma>s$. 
By \eqref{comp20}, for any $\f\in C_c(\Omega)$ with $\|\f\|_{L^\infty}\le 1$, we have
\begin{multline*}
|\langle \mu_n\ast\rho_{\ep_n^s} -  \mu_{n}(t_n^{\sigma})\ast\rho_{\ep_n^s} ,\f\rangle|= |\langle \mu_n - \mu_{n}(t_n^{\sigma}) ,\f\ast\rho_{\ep_n^s}\rangle|\\
\le C\| \mu_n -  \mu_{n}(t_n^{\sigma})\|_{\mathrm{flat}} \,\ep_n^{-s}\le C_l |\log\e_n|^{\sigma +1}  \ep_n^{\sigma-s}\to 0\qquad\textrm{as }n\to +\infty\,.
\end{multline*}
This fact combined with \eqref{convdeb0} implies \eqref{finecomp1}.

Finally we prove \eqref{finecomp2}.
Let $s\in (0,1)\setminus S$ and let $k$ be such that $s\in (s_{k-1},s_k)$. Moreover, let $\sigma\in (s,s_k)$ and set $\hat t_n^\sigma:=\frac{1}{\sr_n^{1-\sigma}}-1$.
 By construction, $(\mu_n - \mu_{n}(\hat t^\sigma_n))(B) = 0$  for any  $B\in \cs_{n}(\hat t^\sigma_n)$.
Therefore, by applying Lemma \ref{speriamo} (ii) with $\Bnew=\bs_{n}(\hat t_{n}^\sigma)$, $\alpha=0$, 
$\displaystyle \beta=\mu_n\res\cup_{\newatop{B\in \cs_{n}(\hat t_{n}^{\sigma})}{\mu_n(B)= 0}}$, for any open set $A\subset\subset\Omega$ and   for $n$ large enough we obtain
\begin{multline}\label{noia10}
\Bigg|\sum_{\newatop{B\in \cs_{n}(\hat t_{n}^\sigma)}{\mu_n(B)=0}}(\mu_n\res B)\ast\rho_{\e_n^s}+
\sum_{{B\in \bs_{n}(\hat t_{n}^\sigma)\setminus  \cs_{n}(\hat t_{n}^\sigma)}}(\mu_n\res B)\ast\rho_{\e_n^s}\Bigg|(A)\\
= \Bigg|\sum_{\newatop{B\in \cs_{n}(\hat t_{n}^\sigma)}{\mu_n(B)=0}}(\mu_n\res B)\ast\rho_{\e_n^s}\Bigg|(A)
\le\sum_{\newatop{B\in \cs_{n}(\hat t_{n}^{\sigma})}{\mu_n(B)= 0}}|(\mu_n\res B)\ast \rho_{\ep_n^s}|(A)\\
\le\sum_{\newatop{B\in \cs_{n}(\hat t_{n}^{\sigma})}{\mu_n(B)= 0}}|(\mu_n\res B)\ast \rho_{\ep_n^s}|
\le C_l\,\|\nabla\rho\|_{L^1}|\log\ep_n| ^{1+\sigma}\ep_n^{\sigma-s}\to 0 \qquad\textrm{ as }n\to +\infty\,,
\end{multline}
where the last inequality follows by  the energy bound \eqref{nuovobounden} and by \eqref{stimarad}. Similarly, one can prove that also
\begin{equation}\label{noia20}
\Bigg|\sum_{\newatop{B\in \cs_{n}(\hat t_{n}^\sigma)}{\mu_n(B)=0}}(\mu_n\res B)\ast\rho_{\sr_{n}^s}+
\sum_{{B\in \bs_{n}(\hat t_{n}^\sigma)\setminus  \cs_{n}(\hat t_{n}^\sigma)}}(\mu_n\res B)\ast\rho_{\sr_{n}^s}\Bigg|(A) \to 0\,.
\end{equation}
Moreover, by arguing as in \eqref{follia} and by using \eqref{ipotesifondam}, it is easy to see that,
for $n$ large enough, 
\begin{equation}\label{follia20}
\begin{aligned}
&\Bigg| \sum_{\newatop{B\in \cs_{n}(\hat t_{n}^\sigma)}{\mu_n(B)\neq 0}}  (\mu_n\res B)\ast\rho_{\e_n^s}\Bigg|= \sum_{\newatop{B\in \cs_{n}(\hat t_{n}^\sigma)}{\mu_n(B)\neq 0}} \Big| (\mu_n\res B)\ast\rho_{\e_n^s}\Big|\,,\\
&\Bigg| \sum_{\newatop{B\in \cs_{n}(\hat t_{n}^\sigma)}{\mu_n(B)\neq 0}}  (\mu_n\res B)\ast\rho_{\sr_{n}^s}\Bigg|= \sum_{\newatop{B\in \cs_{n}(t_{n}^\sigma)}{\mu_n(B)\neq 0}} \Big| (\mu_n\res B)\ast\rho_{\sr_{n}^s}\Big|\,.
\end{aligned}
\end{equation}
Let now $\f\in C_c(\Omega)$ with $\|\f\|_{L^\infty}\le 1$. By \eqref{follia20} and by triangular inequality, we have 
\begin{multline} \label{noia30}
\Bigg|\langle \Big| \sum_{\newatop{B\in \cs_{n}(\hat t_{n}^\sigma)}{\mu_n(B)\neq 0}} (\mu_n\res B)\ast\rho_{\e_n^s}\Big| -\Big| \sum_{\newatop{B\in \cs_{n}(\hat t_{n}^\sigma)}{\mu_n(B)\neq 0}} (\mu_n\res B)\ast \rho_{\sr_{n}^s}\Big|,\f\rangle\Bigg|
\\ 
\le \sum_{\newatop{B\in \cs_{n}(\hat t_{n}^\sigma)}{\mu_n(B)\neq 0}} \langle|(\mu_n\res B)\ast\rho_{\e_n^s}-(\mu_n\res B)\ast \rho_{\sr_{n}^s}|,|\f|\rangle
\\
\le\sum_{\newatop{B\in \cs_{n}(\hat t_{n}^\sigma)}{\mu_n(B)\neq 0}} \langle|(\mu_n\res B-(\mu_n)(B)\delta_{x_B})\ast\rho_{\e_n^s}|,|\f|\rangle
+ \sum_{\newatop{B\in \cs_{n}(\hat t_{n}^\sigma)}{\mu_n(B)\neq 0}}  \langle|(\mu_n\res B-\mu_n(B)\delta_{x_B})\ast\rho_{\sr_{n}^s}|,|\f|\rangle\\
+ \sum_{\newatop{B\in \cs_{n}(\hat t_{n}^\sigma)}{\mu_n(B)\neq 0}} \langle|\mu_n(B)\delta_{x_B}\ast\rho_{\e_n^s}-\mu_n(B)\delta_{x_B}\ast\rho_{\sr_{n}^s}|,|\f|\rangle\to 0\qquad\textrm{as }n\to +\infty\,,
\end{multline}
where the convergence to zero of the first two addenda can be proven by applying Lemma \ref{speriamo} (ii) whereas the convergence to zero of the last addendum follows by Lemma \ref{speriamo} (iii) and the triangular inequality.

Then \eqref{finecomp2} follows by \eqref{noia10}, \eqref{noia20} and \eqref{noia30}.
\vskip2mm
\noindent
\emph{\textbf{Step 2: Proof of  the ${\bf \Gamma}$-liminf inequality.}} 
We can assume without loss of generality that the upper bound \eqref{eneboundcra} is satisfied, so that all the convergences in \eqref{teovecchio1} hold true.
Therefore, by Remark \ref{disuene}, \eqref{stimarad}, \eqref{nuovobounden} and Theorem \ref{mains1} (iv)   we have
$$
\liminf_{n\to +\infty} \frac{\E_{\en}(\mu_n)}{|\log \en|}   
\ge  \liminf_{n\to +\infty} \frac{F(\bs_{n}, \mu_{n}, \Om)}{|\log \sr_{n}|}     
\ge
 \pi\sum_{k=1}^{\sharp S}(s_k-s_{k-1})\nu^{s_{k-1}}(\Omega)\,. 
$$
\vskip2mm
\noindent
\emph{\textbf{Step 3: Proof of  the ${\bf \Gamma}$-limsup inequality.}} Let $\mu$, $\{\nu_s\}$ and $S$ be as in the assumptions.
 By standard density arguments we can assume that $K:=\sharp S<+\infty$, so that $S=\{0=s_0<s_1<\ldots<s_K=1\}$. We set $\eta_\df^{s_0}:=\xi_{\df}^{s_0}$ and 
$\eta_\df^{s_k}:=\nu^{s_k}-\nu^{s_{k-1}}$ for any $k=1,\ldots,K$. It is easy to see that 
\begin{equation}\label{nuovogammalim}
 \sum_{k=1}^{K}(s_k-s_{k-1})\nu^{s_{k-1}}(\Omega)=|\mu|(\Omega)+\sum_{k=0}^{K}(1-s_k)\eta_{\df}^{s_k}(\Omega)\,.
\end{equation}
 Again by density arguments we can assume that 
 $\mu=\sum_{i=1}^N z_i\delta_{x_i}$ with $|z_i|=1$ and $x_i\neq x_j$ for $i\neq j$, and that  $\eta_{\df}^{s_k}=2 \delta_{y^{s_k}}$ with 
all $y^{s_k}$ different from each other and from all $x_i$'s. For any $n\in\N$, we define
$$
\mu_n:=\mu+\sum_{k=0}^K  (\delta_{y^{s_k}_{n,+}}- \delta_{y^{s_k}_{n,-}})\,,
$$
where ${y^{s_k}_{n,+}}$ and ${y^{s_k}_{n,-}}$ are two points in $\Omega$ such that  $\di({y^{s_k}_{n,+}},{y^{s_k}_{n,-}})=2\,\di({y^{s_k}_{n,+}},{y^{s_k}})=2\,\e_n^{s_k}$ for any $k=1,\ldots,K$, while for $k=0$ we consider a dipole whose length tends to zero  slower than any power of $\en$; for instance  such that  $\di({y^{s_0}_{n,+}},{y^{s_0}_{n,-}})=2\,\di({y^{s_0}_{n,+}},{y^{s_0}})=2\,\frac{1}{|\log \e_n|}$.
It is immediate to see that, for $n$ large enough $\mu_n\in X_{\en}(\Omega)$,
$\mu_n\flacon\mu$, $\mu_n\ast\rho_{\e_n}^s\weakstar\mu$ for any $s\in (0,1)$ and $|\mu_n\ast\rho_{\e_n}^s|\weakstar\nu^{s}$ for any $s\in (0,1)\setminus S$.

For any $\xi\in\R^2$, consider the standard polar coordinates centered at $\xi$ and let $\theta_{\xi}$ denote the phase, namely the angular coordinate.
We set
\begin{equation}\label{fasen}
\vartheta_n(\cdot):= \sum_{i=1}^M z_i\theta_{x_i}(\cdot)+\sum_{k=0}^K(\theta_{y^{s_k}_{n,+}}(\cdot)-
\theta_{y^{s_k}_{n,-}}(\cdot))
\end{equation}
and $u_n:=e^{i\vartheta_n}$. Trivially, $u_n\in\AS_{\en}(\mu_n)$ and $\int_{\Omega_{\en}{(\mu_n)}}|\nabla u_n|^2\ud x=\int_{\Omega_{\en}{(\mu_n)}}|\nabla \vartheta_n|^2\ud x$\,. 

Now we show that the pair $(\mu_n,u_n)$ is a recovery sequence for the $\Gamma$-limsup inequality \eqref{limsup}.
Recalling \eqref{nuovogammalim},  we have to prove that
\begin{equation}\label{newgammalimsup}
\limsup_{n\to +\infty}\frac{1}{2|\log\en|}\int_{\Omega_{\en}{(\mu_n)}}|\nabla \vartheta_n|^2\ud x\le \pi|\mu|(\Omega)+\pi\sum_{k=0}^{K}(1-s_k)\eta_{\df}^{s_k}(\Omega)\,.
\end{equation}
 
Fix $r>0$ such that the balls $B_r(x_i)$ and $B_r(y^{s_k})$ are pairwise disjoint and compactly contained in $\Omega$. 
Then 
\begin{equation}\label{zeresima}
\Omega_\en(\mu_n)=\Omega_r(\mu_n)\cup\bigcup_{i=1}^N (B_r(x_i)\setminus B_\en(x_i))\cup\bigcup_{k=0}^K \Big(B_{r}(y^{s_k})\setminus \big(B_{\en}(y_{n,+}^{s_k})\cup B_{\en}(y_{n,-}^{s_k})\big)\Big).
\end{equation}

By construction
\begin{equation}\label{prima}
\frac{1}{2}\int_{\Omega_{r}{(\mu_n)}}|\nabla u_n|^2\ud x\le C_r,
\end{equation}
where $C_r$ is a constant independent  of $n$.
Moreover, for any $i=1,\ldots,M$, by using Cauchy-Schwarz inequality together with the fact that $|\nabla\theta_\xi(\cdot)|= \frac{1}{|\cdot -\xi|}$, we deduce 
\begin{equation}\label{seconda}
\frac 1 2\int_{B_{r}(x_i)\setminus B_\en(x_i)}|\nabla\vartheta_n|^2\ud x\le \frac 1 2\int_{B_{r}(x_i)\setminus B_\en(x_i)}|\nabla\theta_{x_i}|^2\ud x+ C=\pi\log\frac {r}{\en}+C\,,
\end{equation}
for some positive constant $C>0$.
Arguing analogously, one can easily check that for any $k=0,1,\ldots,K$  there holds
\begin{multline}\label{terza}
\frac 1 2\int_{B_{\e_n^{s_k}}(y^{s_k}_{n,\pm})\setminus B_\en(y^{s_k}_{n,\pm})}|\nabla\vartheta_n|^2\ud x\le\frac 1 2\int_{B_{\e_n^{s_k}}(y^{s_k}_{n,\pm})\setminus B_\en(y^{s_k}_{n,\pm})}|\nabla\theta_{y^{s_k}_{n,\pm}}|^2\ud x\\
+\frac 1 2\int_{B_{\e_n^{s_k}}(y^{s_k}_{n,\pm})\setminus B_\en(y^{s_k}_{n,\pm})}|\nabla\theta_{y^{s_k}_{n,\mp}}|^2\ud x+
\int_{B_{\e_n^{s_k}}(y^{s_k}_{n,\pm})\setminus B_\en(y^{s_k}_{n,\pm})}|\nabla\theta_{y^{s_k}_{+}}|\,|\nabla\theta_{y^{s_k}_{-}}|\ud x
+C\\
\le \pi(1-s_k)|\log\en|+C\,.
\end{multline} 

Furthermore, by straightforward computations, it follows that for any $k=0,1,\ldots,K$ 
\begin{multline}\label{quarta}
\frac 1 2\int_{B_{r}(y^{s_k})\setminus B_{8\e_n^{s_k}}(y^{s_k})}|\nabla\vartheta_n|^2\ud x\le \frac 1 2\int_{B_{r}(y^{s_k})\setminus B_{8\e_n^{s_k}}(y^{s_k})}|\nabla (\theta_{y^{s_k}_{n,+}}-\theta_{y^{s_k}_{n,-}})|^2\ud x+C
\le C.
\end{multline}

Finally, by scaling arguments, it is easy to see that for any $k=0,1,\ldots,K$
\begin{equation}\label{quinta}
\frac 1 2\int_{B_{8\e^{s_k}_n}(y^{s_k})\setminus (B_{\e^{s_k}_n}(y^{s_k}_{+})\cup B_{\en}(y^{s_k}_{-}))}|\nabla\vartheta_n|^2\ud x \le C\,.
\end{equation}

By summing \eqref{prima}, \eqref{seconda}, \eqref{terza}, \eqref{quarta}, \eqref{quinta}, and recalling \eqref{zeresima}, we obtain \eqref{newgammalimsup}, i.e., the claim. 


\end{proof}
\section{$\Gamma$-convergence of $GL_\e$}
This section is devoted to the $\Gamma$-convergence analysis of the Ginzburg-Landau functionals \eqref{GL}.

\begin{theorem}\label{mainsGL} 
Let $\{\e_n\}\subset\R^+$ with $\e_n\to 0$ as $n\to +\infty$.  The following $\Gamma$-convergence result holds true. 
\begin{itemize}
\item[] (\textbf{Compactness})
Let $\{u_n\} \subset H^1(\Om;\R^2)$  be such that 
\begin{equation}\label{eneboundgl}
 GL_{\ep_n}(u_n) \le C |\log \e_n| \qquad \text{ for all } n\in \N,
\end{equation}
for some  constant $C>0$ independent of $n$.
Set $J^s u_n:=Ju_n\ast\rho_{\e_n^s}$ for any $n\in\N$ and for any $0<s<1$. 

Then, there exist a measure $\mu\in X(\Omega)$,
a countable (finite or infinite) set $S:=\{0=s_0<s_1<s_2<\ldots\}$ with $\sup S=1$, and 
a family  of positive Radon measures $\{\nu^s\}_{s\in [0,1)} \subset X(\Omega)$,  such that  the following compactness and structure properties hold true:
\begin{itemize}
\item[(1)] \it{Flat and weak star compactness for the Jacobians}: Up to a subsequence $Ju_n\flacon \pi \mu$ and
for any $s\in (0,1)$,  $J^s u_n \weakstar \pi\mu$ up to a  subsequence (independent of $s$).
\item[(2)] \it{Weak star compactness for vorticity densities}: For any $s\in (0,1)\setminus S$, $|J^s u_n|\weakstar \pi \nu^s$  up to a  subsequence (independent of $s$).
\item[(3)] \it{Structure of vorticity densities}: 
For all $s\in[0,1)$,  $\nu^s= |\mu| + 2 \xi_{\mathrm{def}}^s$ for some {\it defect measure}  $\xi^s_{\mathrm{def}}\in X(\Om)$ with $\xi^s_{\mathrm{def}}\ge 0$. 
Moreover, $\nu^s$  is constant in $[s_{k-1},s_{k})$ and
$\nu^{s_{k-1}} \le \nu^{s_{k}}$ for any $k=1,\ldots,\sharp S$.
\end{itemize}
\item[] (\textbf{${\bf \Gamma}$-liminf inequality}) 
 Let $\{u_n\} \subset H^1(\Om;\R^2)$ be such that for all $s\in (0,1)$, $|J^s u_n|\weakstar \pi \nu^s$ for some family of measures $\{\nu_s\}_{s\in[0,1)}$ as in (3). Then, 
$$
\displaystyle \liminf_{n\to +\infty} \frac{GL_{\e_n}(u_n)}{|\log\e_n|}\ge\pi\sum_{k=1}^{\sharp S}(s_k-s_{k-1})\nu^{s_{k-1}}(\Omega)\,.
$$

\item[] (\textbf{${\bf \Gamma}$-limsup inequality}) For any $\mu\in X(\Omega)$ and for any
family of measures $\{\nu^s\}_{s\in [0,1)}$ as in (3)  with 
$\sum_{k=1}^{\sharp S}(s_k-s_{k-1})\nu^{s_{k-1}}(\Omega) <+\infty$, there exists $\{u_n\} \subset H^1(\Om;\R^2)$ 
such that $Ju_n\flacon \pi \mu$, $J^s u_n\weakstar \pi \mu$ for any $s\in (0,1)$, $|J^s u_n|\weakstar \pi \nu^s$ for any $s\in (0,1)\setminus S$, and
\begin{equation}\label{limsupgl}
\displaystyle \lim_{n\to +\infty} \frac{GL_{\e_n}(u_n)}{|\log\e_n|} = \pi\sum_{k=1}^{\sharp S}(s_k-s_{k-1})\nu^{s_{k-1}}(\Omega)\,.
\end{equation}
\end{itemize}
\end{theorem}
\begin{proof}
First we will prove the compactness properties and the $\Gamma$-liminf inequality, following the approach in \cite{S, AP}. By standard density arguments in $\Gamma$-convergence we may assume that the functions $u_n$ are smooth. 
\vskip2mm
\noindent
\emph{\textbf{Step 1: Energy estimates.}}
 Following the notations in \cite{S,AP}, for any $n\in\N$ and for any $\tau \in (0,1)$ we set  
\begin{eqnarray*}
\Om_{n,\tau}:= \{|u_n| > \tau\}, \qquad \gamma_{n,\tau}:= \partial \Om_{n,\tau}\setminus \partial \Om, \qquad K_{n,\tau}:=\Omega\setminus\Omega_{n,\tau},\\
\Theta_{n}(\tau):= \frac12 \int_{\Om_{n,\tau}} \left|\nabla\frac{u_n}{|u_n|}\right|^2\ud x, \qquad m_{n}(\tau):= \int_{\gamma_{n,\tau}} |\nabla |u_n||\ud\hs^1.
\end{eqnarray*}
By the Coarea Formula we have
\begin{equation}\label{GLcoarea}
GL_{\e_n}(u_n) \ge \frac12 \int_0^\infty \left( m_n(\tau) + \frac{2 W(\tau)}{\e_n^2}\int_{\gamma_{n,\tau}} \frac{1}{ |\nabla |u_n||}\ud\hs^1\right) \ud\tau - \int_0^\infty \tau^2 \ud \Theta_n'(\tau) ,
\end{equation}
where  $\Theta_n'(\tau)$ is the distributional derivative of the decreasing function $\Theta_n(\tau)$ and the inequality is due to the possible presence of flat regions $\{\nabla |u_n| = 0\}$ with positive measure.   
\par 

Fix $l\in\N$ (large). For any $\tau \le 1-\frac{1}{l}$ and $n$ large enough 
we have 
$
|K_{n,\tau}|\le C_l \, \e_n^{2} |\log\e_n| \le \frac{1}{2}|\Om|$, so that, using also that $\Om$ is Lipschitz we have 
\begin{equation*}
\hs^1(\partial K_{n,\tau})\le C_l \hs^1(\gamma_{n,\tau}). 
\end{equation*}
Notice that, by definition of Hausdorff measure, 
since $\partial K_{n,\tau}$ is compact, it is always contained in a finite  union of balls $B_{r_i}(y_i)$ such that $\sum_i r_i \le  \hs^1(\partial K_{n,\tau})$. Moreover, after a merging procedure, we can always assume that such balls are disjoint.  As a consequence, either $K_{n,\tau}$ or $\Om\setminus K_{n,\tau} = \Om_{n,\tau}$ is  contained in the union of such balls. In the latter case, since  $|\Om\setminus K_{n,\tau}|\ge  \frac{1}{2}|\Om|$ we have $\sum_i r_i \ge C$, and we  replace these balls by  one single ball containing $\Om$. In both cases, we have a family of balls 
$\bs_{n,\tau}$ whose union contains $K_{n,\tau}$,  such that 
\begin{equation}\label{iso2}
\rad (\bs_{n,\tau})\le C \hs^1(\partial K_{n,\tau}) \le C_l \hs^1(\gamma_{n,\tau}). 
\end{equation}
Notice that we can always assume (just by enlarging an arbitrarily chosen ball in $\bs_{n,\tau}$) that
\begin{equation*}
\en \le \rad (\bs_{n,\tau})\le  C_l \hs^1(\gamma_{n,\tau}) +\en. 
\end{equation*}
Since $ K_{n,\tau}$ is monotone in $\tau$ (with respect to inclusion), for any given $\eta >0$ (possibly depending on $n$) we can always  assume that
the map $\tau \mapsto \rad(\bs_{n,\tau})$ is measurable and  
\begin{equation}\label{disraggi}
\e_n\le\rad(\bs_{n,\tau_1})\le\rad(\bs_{n,\tau_2}) + \eta \qquad \text{ for all }  0<\tau_1< \tau_2 <1-\frac{1}{l}.
\end{equation}
By H\"older inequality  we have
\begin{equation}\label{Ho}
\hs^1(\gamma_{n,\tau})^2 \le C\, m_n(\tau) \int_{\gamma_{n,\tau}} \frac{1}{|\nabla |u_n||}\ud\hs^1.
\end{equation}
Fix $\bar \tau \in (\frac 2l, 1-\frac{2}{l})$.  By using, in order of appearance, \eqref{GLcoarea}, \eqref{Ho}, Young inequality, integration by parts of $\tau^2 \ud\Theta_n'(\tau)$, and \eqref{iso2}, we obtain
\begin{eqnarray}
\nonumber
GL_{\e_n}(u_n) \ge \frac12 \int_0^{\infty}  m_n(\tau) +  \frac{\hs^1(\gamma_{n,\tau})^2 W(\tau)}{C \e_n^2 m_n(\tau)}\ud \tau      -  \int_0^{\infty}  \tau^2 \ud\Theta_n'(\tau)  
\\ \nonumber
\ge \int_0^{\infty} \frac{\sqrt{W(\tau)}}{ C \e_n} \hs^1(\gamma_{n,\tau})+ 2 \tau\Theta_n(\tau) \ud \tau
\\ \label{Modica0}
\ge \int_{0}^{1-\frac 1l} \frac{\sqrt{W(\tau)}}{ C_l \e_n} (\rad(\bs_{n,\tau}) -\en)     + 2 \tau\Theta_n(\tau) \ud \tau 
\\ \label{Modica}
\ge \int_{0}^{\bar \tau} \frac{\sqrt{W(\tau)}}{ C_l \e_n} (\rad(\bs_{n,\tau}) -\en)     + 2 \tau\Theta_n(\tau) \ud \tau 
\\
\label{Modica2}
+ \int_{ \bar\tau}^{\bar\tau + \frac 1l} \frac{\sqrt{W(\tau)}}{ C_l \e_n} (\rad(\bs_{n,\tau}) - \en) + 2 \tau\Theta_n(\tau) \ud \tau\,.
\end{eqnarray}
By \eqref{disraggi}, \eqref{Modica2} and the energy bound \eqref{eneboundgl}, 
 it follows that, for $\eta$ small enough (depending on $n$), 
\begin{equation}\label{pelo}
\en \le \rad(\bs_{n,\bar\tau})  \le  C_l  (\e_n|\log\e_n|  + \en+  \eta) \le  C_l\, \e_n|\log\e_n|\,.
\end{equation} 

Let $\bs_{n,\bar\tau}(t)$ be a time parametrized family of balls given by  Proposition \ref{ballconstr} starting from $\bs_{n,\bar\tau}=:\bs_{n,\bar\tau}(0)$; we denote by $\cs_{n,\bar\tau}(t)$  the family of balls in $\bs_{n,\bar\tau}(t)$ that are contained in $\Om$ and by $U_{n,\bar\tau}(t)$ the union of the balls in $\bs_{n,\bar\tau}(t)$.
Moreover we set
$$
\mu_{n,\bar\tau}(t) := \sum_{B \in \cs_{n,\bar\tau}(t)}  \deg(u_n, \partial B) \delta_{x_B}\,.
$$
Let $F$ be as defined in \eqref{minf}. 
  By \eqref{sti3} in Proposition \ref{ballconstr}, for any $t_n\ge 1$ we have
 \begin{equation}\label{pelo2}
 F(\bs_{n,\bar\tau}(0),\mu_{n,\bar\tau}(0),\Omega\cap(U_{n,\bar\tau}(t_n)\setminus U_{n,\bar\tau}(0)))\ge \pi|\mu_{n,\bar\tau}(t_n)|\log (1+t_n)\,.
  \end{equation}
Moreover, \eqref{Modica} implies in particular that 
$$
GL_{\e_n}(u_n) \ge 
 \int_0^{\bar\tau}  2 \tau\Theta_n(\tau) \ud \tau \ge C_l \Theta_n(\bar\tau)\,,
$$
which together with  \eqref{pelo2},  Remark \ref{disuene}, the energy bound  and \eqref{pelo} yields
 \begin{multline}\label{sommar}
 \log(1+t_n)|\mu_{n,\bar\tau}(t_n)| \le C  F(\bs_{n,\bar\tau}(0),\mu_{n,\bar\tau}(0),\Omega) \le C \Theta_{n}(\bar\tau) \\
 \le C_l |\log \e_n| \le C_l |\log \rad(\bs_{n,\bar\tau})|\,.
 \end{multline}
 In particular, by applying \eqref{sommar} with $t_n=1$, we  have
 \begin{equation}\label{boundfprel}
|\mu_{n,\bar\tau}(1)|   \le C_l\,F(\bs_{n,\bar\tau}(0),\mu_{n,\bar\tau}(0),\Omega)\le C_l |\log \rad(\bs_{n,\bar\tau})|\,.
  \end{equation}
Moreover, by construction (see also \eqref{minf}),
$$
F(\bs_{n,\bar\tau}(1),\mu_{n,\bar\tau}(1),\Omega)\le F(\bs_{n,\bar\tau}(0),\mu_{n,\bar\tau}(0),\Omega)\le  C_l |\log \rad(\bs_{n,\bar\tau})|\,,
$$
which, in virtue of \eqref{boundfprel} and Proposition \eqref{ballconstr} (5), implies
 \begin{equation}\label{boundf}
F(\bs_{n,\bar\tau}(1),\mu_{n,\bar\tau}(1),\Omega)+|\mu_{n,\bar\tau}(1)| \le C_l |\log \rad(\bs_{n,\bar\tau})| \le C_l |\log \rad(\bs_{n,\bar\tau}(1))| \,.
  \end{equation}

\vskip2mm
\noindent
\emph{\textbf{Step 2: Proof of  compactness.}} 
Let $\bar\tau\in(\frac  2 l,1-\frac 1 l)$.
By \eqref{boundf} and by Theorem \ref{mains1}, there exist $\mu_{\bar\tau}\in X(\Om)$ and $\{\nu_{\bar\tau}^s\}_{s\in [0,1)}\subset X(\Om)$ such that, up to a subsequence independent of $s$,
\begin{equation}\label{comp1}
 \mu_{n,{\bar\tau}}(1) \flacon \mu_{\bar\tau},\quad  \mu_{n,{\bar\tau}}^s(1)\weakstar\mu_{\bar\tau}\qquad\textrm{for any }s\in (0,1)
\end{equation}
and
\begin{equation*}
|\mu_{n,{\bar\tau}}^s(1)|\weakstar\nu^s_{\bar\tau}\qquad\textrm{for any }s\in (0,1)\setminus S_{\bar\tau},
\end{equation*}
where $S_{\bar\tau}$ is the set constructed in the proof of Theorem \ref{mains1}. Moreover  the measures $\nu_{\bar\tau}^s$ satisfy all the structure properties in (3).

Firstly, we show that, up to a subsequence, $Ju_n\flacon \pi\mu_{\bar\tau }$ and that actually  $\mu_{\bar\tau}$ does not depend on $\bar\tau$. 
By construction and by the very definition of $J_{\bar\tau}u_n$ (see \eqref{defcur} and \eqref{moja}) we have that for any $t_n\ge 1$
\begin{equation*}
(J_{\bar\tau}u_n - \pi \mu_{n,\bar\tau}(t_n))(B) = 0\qquad\textrm{ for any }\quad B\in \cs_{n,\bar\tau}(t_n)\,.
\end{equation*} 
Therefore, by triangular inequality, Proposition \ref{sempre}, Lemma \ref{speriamo} (i), Proposition \ref{ballconstr} (5), \eqref{pelo} and \eqref{boundf}, for any $t_n\ge 1$ and for $n$ large enough we have
\begin{multline}\label{comp2}
\|Ju_n - \pi \mu_{n,\bar\tau}(t_n)\|_{\mathrm{flat}} \le \|Ju_n - J_{\bar\tau}u_n\| _{\mathrm{flat}}+\|J_{\bar\tau}u_n-\pi \mu_{n,\bar\tau}(t_n)\|_{\mathrm{flat}}\\
\le C_l\ep_n|\log\ep_n| +C\,\rad(\bs_{n,\bar\tau}(t_n))\,(|J_{\bar\tau}u_n|(\Omega)+\pi|\mu_{n,\bar\tau}(t_n)|(\Omega))\\
\le C_l\ep_n|\log\ep_n|+
 C_l \,\rad( \bs_{n, \bar\tau}(t_n))  (|\log\e_n|+|\mu_{n,\bar\tau}(1)|(\Omega)) \\
 \le C_l |\log\e_n|^2 (1+t_n)\,\ep_n\,.
\end{multline}
Therefore, by \eqref{comp1}, applying \eqref{comp2} with $t_n= 1$ and  setting $\mu:=\mu_{\bar\tau}$, we get
\begin{equation}\label{comp3}
Ju_n\flacon\pi\mu, \qquad \mu_{n,\bar\tau}(1)\flacon\mu\qquad\textrm{for any }\bar\tau\in\left(\frac 2 l, 1-\frac 2 l\right), 
\end{equation}
up to a subsequence independent of $\bar\tau$.

\vskip2mm
We now prove that $J^s u_n\equiv Ju_n\ast \rho_{\ep_n^s}\weakstar \pi\mu$ for any $s\in (0,1)$ up to a subsequence independent of $s$. To this end, fix $s\in (0,1)$ and let $\sigma>s$.  By applying \eqref{comp2}  with $t_n=t_n^\sigma:=\frac{1}{\ep_n^{1-\sigma}}-1$, for any $\f\in C_c(\Omega)$ with $\|\f\|_{L^\infty}\le 1$, we get
\begin{multline}\label{convdeb}
|\langle Ju_n\ast\rho_{\ep_n^s} - \pi \mu_{n,\bar\tau}(t_n^{\sigma})\ast\rho_{\ep_n^s} ,\f\rangle|= |\langle Ju_n - \pi \mu_{n,\bar\tau}(t_n^{\sigma}) ,\f\ast\rho_{\ep_n^s}\rangle|\\
\le C\| Ju_n - \pi \mu_{n,\bar\tau}(t_n^{\sigma})\|_{\mathrm{flat}} \,\ep_n^{-s}\le C_l |\log\e_n|^2 \ep_n^{\sigma-s}\,.
\end{multline}
By using \eqref{comp2} with $t_n=t_n^\sigma$ and by \eqref{comp3}, $\mu_{n,\bar\tau}(t_n^{\sigma})\flacon\mu$; moreover,  by \eqref{sommar},
$|\mu_{n,\bar\tau}(t_n^\sigma)|\le\frac{C_l}{1-\sigma}$, so that
 $\mu_{n,\bar\tau}(t_n^{\sigma})\weakstar\mu$ and $\mu_{n,\bar\tau}(t_n^{\sigma})\ast\rho_{\ep_n^s} \weakstar\mu$.
This latter fact, combined with \eqref{convdeb}  implies that $Ju_n\ast \rho_{\ep_n^s}\weakstar \pi\mu$, as claimed above.

\vskip2mm
In order to conclude the proof of (i), we  show that, up to a subsequence independent of $s$, $|Ju_n\ast \rho_{\e_n^s}| \weakstar \pi\nu_{\bar\tau }^s$ for any $s\in(0,1)\setminus S_{\bar\tau}$ and that actually  $\nu_{\bar\tau }^s$ 
do not depend on $\bar\tau$ and, in turn, $S_{\bar\tau}$ can be also chosen  independent of  $\bar\tau$.
We denote by $s_k$ the points in the set $S_{\bar\tau}$ with $s_0=0$ and $s_{k-1}< s_k$ for any $k=1,\ldots,\sharp S$.
Set  $\sr_{n,\bar\tau}:=\rad(\bs_{n,\bar\tau})$ and $t_{n,\bar\tau}^\sigma:=\frac{1}{\sr_{n,\bar\tau}^{1-\sigma}}-1$ for any $\sigma\in (0,1)$. We recall that,   by \eqref{weakconv}, if $s\in (s_{k-1},s_k)$ for some $k$, then for any  $\sigma\in (s,s_k)$
\begin{equation}\label{comeweakconv}
|\mu_{n,\bar\tau}(t_{n,\bar\tau}^\sigma)|\weakstar \nu_{\bar\tau}^s\qquad\textrm{as }n\to +\infty\,.
\end{equation}

By triangular inequality,  to prove the claim 
 it is enough to show that for any $s\in(0,1)\setminus S_{\bar\tau}$, (up to a subsequence independent of $s$) there holds
\begin{eqnarray}\label{uno}
|Ju_n\ast \rho_{\e_n^s}|- |J_{\bar\tau}u_n\ast \rho_{\e_n^s}|\weakstar 0\qquad\textrm{ as }n\to +\infty\, ,\\ \label{due} 
|J_{\bar\tau}u_n\ast \rho_{\e_n^s}|-|J_{\bar\tau}u_n\ast \rho_{\sr_{n,\bar\tau}^s}| \weakstar 0 \qquad\textrm{ as }n\to +\infty \, , \\ \label{tre}
|J_{\bar\tau}u_n\ast \rho_{\sr_{n,\bar\tau}^s}| \weakstar \nu_{\bar\tau}^s\qquad\textrm{ as }n\to +\infty\,.
\end{eqnarray}

As for the proof of \eqref{uno}, notice that, by Proposition \ref{sempre} and the energy bound, for any $\f\in C_c(\Omega)$ with $\|\f\|_{L^\infty}\le 1$ we have
\begin{multline*}
\left|\langle|Ju_n\ast \rho_{\e_n^s}|- |J_{\bar\tau}u_n\ast \rho_{\e_n^s}|,\f\rangle\right|\le \langle |(Ju_n- J_{\bar\tau}u_n)\ast \rho_{\e_n^s}|,|\f|\rangle\le \langle |Ju_n- J_{\bar\tau}u_n|\ast \rho_{\e_n^s},|\f|\rangle\\
\le  \|Ju_n- J_{\bar\tau}u_n\|_{\flap}\,\|\nabla(|\f|\ast \rho_{\e_n^s})\|_{L^\infty}\le C_l \ep_n^{1-s}|\log\ep_n|\to 0 \qquad\textrm{ as }n\to +\infty\,.
\end{multline*}
Let us pass to the proof of \eqref{due}.
Let $s\in (0,1)\setminus S_{\bar\tau}$, let $k$ be such that $s\in (s_{k-1},s_k)$, and let $\sigma\in (s,s_k)$.
Let $A\subset\subset\Omega$ be open. 
By arguing as in  \eqref{noia10} and \eqref{noia20} and replacing $\mu_n$ with $J_{\bar\tau}u_n$, we get
\begin{eqnarray}\label{noia1}
\quad\Bigg|\sum_{\newatop{B\in \cs_{n,\bar\tau}(t_{n,\bar\tau}^\sigma)}{J_{\bar\tau}u_n(B)=0}}(J_{\bar\tau}u_n\res B)\ast\rho_{\e_{n,\bar\tau}^s}+
\sum_{{B\in \bs_{n,\bar\tau}(t_{n,\bar\tau}^\sigma)\setminus  \cs_{n,\bar\tau}(t_{n,\bar\tau}^\sigma)}}(J_{\bar\tau}u_n\res B)\ast\rho_{\e_{n,\bar\tau}^s}\Bigg|(A) \to 0\,,\\ \label{noia2}
\quad\Bigg|\sum_{\newatop{B\in \cs_{n,\bar\tau}(t_{n,\bar\tau}^\sigma)}{J_{\bar\tau}u_n(B)=0}}(J_{\bar\tau}u_n\res B)\ast\rho_{\delta_{n,\bar\tau}^s}+
\sum_{{B\in \bs_{n,\bar\tau}(t_{n,\bar\tau}^\sigma)\setminus  \cs_{n,\bar\tau}(t_{n,\bar\tau}^\sigma)}}(J_{\bar\tau}u_n\res B)\ast\rho_{\delta_{n,\bar\tau}^s}\Bigg|(A) \to 0\,.
\end{eqnarray}
Moreover, by arguing as in \eqref{follia20} and by using the first inequality in \eqref{pelo}, it is easy to see that
\begin{equation}\label{follia2}
\begin{aligned}
&\Bigg| \sum_{\newatop{B\in \cs_{n,\bar\tau}(t_{n,\bar\tau}^\sigma)}{J_{\bar\tau}u_n(B)\neq 0}}  (J_{\bar\tau}u_n\res B)\ast\rho_{\e_n^s}\Bigg|= \sum_{\newatop{B\in \cs_{n,\bar\tau}(t_{n,\bar\tau}^\sigma)}{J_{\bar\tau}u_n(B)\neq 0}} \Big| (J_{\bar\tau}u_n\res B)\ast\rho_{\e_n^s}\Big|,\\
&\Bigg| \sum_{\newatop{B\in \cs_{n,\bar\tau}(t_{n,\bar\tau}^\sigma)}{J_{\bar\tau}u_n(B)\neq 0}}  (J_{\bar\tau}u_n\res B)\ast\rho_{\sr_{n,\bar\tau}^s}\Bigg|= \sum_{\newatop{B\in \cs_{n,\bar\tau}(t_{n,\bar\tau}^\sigma)}{J_{\bar\tau}u_n(B)\neq 0}} \Big| (J_{\bar\tau}u_n\res B)\ast\rho_{\sr_{n,\bar\tau}^s}\Big|\,,
\end{aligned}
\end{equation}
which, by arguing as in \eqref{noia30}, implies that for any $\f\in C_c(\Omega)$ with $\|\f\|_{L^\infty}\le 1$
\begin{equation} \label{noia3}
\Bigg|\langle \Big| \sum_{\newatop{B\in \cs_{n,\bar\tau}(t_{n,\bar\tau}^\sigma)}{J_{\bar\tau}u_n(B)\neq 0}} (J_{\bar\tau}u_n\res B)\ast\rho_{\e_n^s}\Big| -\Big| \sum_{\newatop{B\in \cs_{n,\bar\tau}(t_{n,\bar\tau}^\sigma)}{J_{\bar\tau}u_n(B)\neq 0}} (J_{\bar\tau}u_n\res B)\ast \rho_{\delta_{n,\bar\tau}^s}\Big|,\f\rangle\Bigg|
\to 0\qquad\textrm{as }n\to +\infty,
\end{equation}
Then \eqref{due} follows by \eqref{noia1}, \eqref{noia2} and \eqref{noia3}.

We end up with the proof of (i) by showing \eqref{tre} which, in view of \eqref{comeweakconv} and \eqref{noia2} is equivalent to
\begin{equation}\label{noiainfinita}
\Bigg| \sum_{\newatop{B\in \cs_{n,\bar\tau}(t_{n,\bar\tau}^\sigma)}{J_{\bar\tau}u_n(B)\neq 0}} (J_{\bar\tau}u_n\res B)\ast\rho_{\delta_{n,\bar\tau}^s}\Bigg| -\Bigg| \sum_{\newatop{B\in \cs_{n,\bar\tau}(t_{n,\bar\tau}^\sigma)}{J_{\bar\tau}u_n(B)\neq 0}} (\mu_{n,\bar\tau}(t_{n,\bar\tau}^\sigma)\res B)\Bigg|\weakstar 0\qquad\textrm{as }n\to +\infty\,.
\end{equation}
Let $\f\in C_c(\Omega)$ with $\|\f\|_{L^\infty}\le 1$. By \eqref{follia2} and by triangular inequality we have
\begin{multline*}
\Bigg|\langle \Big| \sum_{\newatop{B\in \cs_{n,\bar\tau}(t_{n,\bar\tau}^\sigma)}{J_{\bar\tau}u_n(B)\neq 0}} (J_{\bar\tau}u_n\res B)\ast\rho_{\delta_{n,\bar\tau}^s}\Big| -\Big| \sum_{\newatop{B\in \cs_{n,\bar\tau}(t_{n,\bar\tau}^\sigma)}{J_{\bar\tau}u_n(B)\neq 0}} (\mu_{n,\bar\tau}(t_{n,\bar\tau}^\sigma)\res B)\Big|,\f\rangle\Bigg|\\
\le
 \sum_{\newatop{B\in \cs_{n,\bar\tau}(t_{n,\bar\tau}^\sigma)}{J_{\bar\tau}u_n(B)\neq 0}}\Big| \langle\Big| (J_{\bar\tau}u_n\res B)\ast\rho_{\delta_{n,\bar\tau}^s}\Big| -\Big| (\mu_{n,\bar\tau}(t_{n,\bar\tau}^\sigma)\res B)\Big|,\f\rangle\Big|
\\
\le 
 \sum_{\newatop{B\in \cs_{n,\bar\tau}(t_{n,\bar\tau}^\sigma)}{J_{\bar\tau}u_n(B)\neq 0}}\Big| \langle\Big| (J_{\bar\tau}u_n\res B)\ast\rho_{\delta_{n,\bar\tau}^s}\Big| -\Big| (\mu_{n,\bar\tau}(t_{n,\bar\tau}^\sigma)\res B)\ast \rho_{\delta_{n,\bar\tau}^s}\Big|,\f\rangle\Big|\\
 +\sum_{\newatop{B\in \cs_{n,\bar\tau}(t_{n,\bar\tau}^\sigma)}{J_{\bar\tau}u_n(B)\neq 0}}\Big| \langle \Big| (\mu_{n,\bar\tau}(t_{n,\bar\tau}^\sigma)\res B)\ast \rho_{\delta_{n,\bar\tau}^s}\Big|-\Big| (\mu_{n,\bar\tau}(t_{n,\bar\tau}^\sigma)\res B)\Big| ,\f\rangle\Big|\to 0\qquad\textrm{ as }n\to +\infty\, ,
\end{multline*}
where the convergence to zero of the first addendum and of the second addendum follow by Lemma \ref{speriamo} (ii) and (iii) respectively.

This concludes the proof of \eqref{noiainfinita},   of \eqref{tre} and, in turn, of all the compactness properties in (i).

\vskip2mm
\noindent
\emph{\textbf{Step 3: Proof of  the ${\bf \Gamma}$-liminf inequality.}} 
We can assume without loss of generality that the upper bound \eqref{eneboundgl} holds true.
Then, by \eqref{Modica0}, Fatou Lemma, Remark \ref{disuene}, \eqref{pelo}, \eqref{boundf} and Theorem \ref{mains1} (iv)   we have
\begin{multline*}
\liminf_{n\to +\infty} \frac{GL_{\en}(u_n)}{|\log \en|}   
\ge 
\liminf_{n\to +\infty} \frac{1}{|\log \en|}  \int_{\frac 2l}^{1-\frac 2l}  2 \tau\Theta_n(\tau) \ud \tau 
\\ \ge
\int_{\frac 2l}^{1-\frac 2l} 2 \tau  \liminf_{n\to +\infty} \frac{\Theta_n(\tau)}{|\log \en|}     \ud \tau 
\ge
\int_{\frac 2l}^{1-\frac 2l} 2 \tau  \liminf_{n\to +\infty} \frac{F(\bs_{n,\tau}(1), \mu_{n,\tau}(1), \Om)}{|\log \sr_{n,\tau}|}     \ud \tau 
\\
\ge
\Big(1-\frac 4l\Big) \, \pi\sum_{k=1}^{\sharp S}(s_k-s_{k-1})\nu^{s_{k-1}}(\Omega),
\end{multline*}
whence the claim follows by sending $l\to\infty$.

\vskip2mm
\noindent
\emph{\textbf{Step 4: Proof of  the ${\bf \Gamma}$-limsup inequality.}} 
The construction of the recovery sequence closely resembles the one in the proof of Theorem \ref{mainscra}.
We briefly sketch it.

Let $\mu$, $\{\nu^s\}$ and $S$ be as in the assumption.
By standard density arguments we can assume that $K:=\sharp S<+\infty$ so that $S:=\{0=s_0<s_1<\ldots<s_K=1\}$. We set $\eta_\df^{s_0}:=\xi_{\df}^{s_0}$ and 
$\eta_\df^{s_k}:=\nu^{s_k}-\nu^{s_{k-1}}$ for any $k=1,\ldots,K$. 
 Again by density arguments we can assume that 
 $\mu=\sum_{i=1}^N z_i\delta_{x_i}$ with $|z_i|=1$ and $x_i\neq x_j$ for $i\neq j$, and that  $\eta_{\df}^{s_k}=2 \delta_{y^{s_k}}$ with 
all $y^{s_k}$ different from each other and from all $x_i$'s.
Let moreover ${y^{s_k}_{n,+}}$ and ${y^{s_k}_{n,-}}$ be two points in $\Omega$ such that  $\di({y^{s_k}_{n,+}},{y^{s_k}_{n,-}})=2\,\di({y^{s_k}_{n,+}},{y^{s_k}})=2\,\e_n^{s_k}$ for any $k=1,\ldots,K$, 
while for $k=0$ we enforce that  $\di({y^{s_0}_{n,+}},{y^{s_0}_{n,-}})=2\,\di({y^{s_0}_{n,+}},{y^{s_0}})=2\,\frac{1}{|\log \e_n|}$.
For any $n\in\N$, let $g_n(r):=\min\{\frac{r}{\en},1\}$ and let $\vartheta_n$ be defined as in \eqref{fasen}, for any $x\in\Omega$ we set 
$$
u_n(x):=\prod_{i=1}^N g_n(|\cdot-x_i|)\cdot\prod_{k=0}^K  g_n(|\cdot-y_{n,+}^{s_k}|)\cdot g_n(|\cdot-y_{n,-}^{s_k}|)\cdot e^{i\vartheta_n(\cdot)}.
$$
One can easily check that $u_n\in H^1(\Omega;\R^2)$ and that $Ju_n$ satisfy all the desired convergence properties. By arguing as in the proof of Theorem \ref{mainscra}, one can prove that
$$
\limsup_{n\to +\infty}\frac{GL_{\en}(u_n)}{|\log\en|}\le  \pi|\mu|(\Omega)+\pi\sum_{k=0}^{K}(1-s_k)\eta_{\df}^{s_k}(\Omega),
$$ 
which, in view of \eqref{nuovogammalim}, and sending eventually $K\to +\infty$, yields \eqref{limsupgl}. 
\end{proof}

\vskip4mm
\textbf{Acknowledgements.}  
The authors gratefully acknowledge Marco Cicalese for fruitful and interesting discussions.

\end{document}